\documentclass{amsart}

\usepackage[english]{babel}
\usepackage[utf8]{inputenc}
\usepackage{amsmath,amsthm}
\usepackage{amsfonts}
\usepackage{graphicx}
\usepackage[colorinlistoftodos]{todonotes}
\usepackage{hyperref}
\usepackage{subfig}

\newtheorem{proposition}{Proposition}
\newtheorem{definition}{Definition}

\theoremstyle{definition}
\newtheorem{example}{Example}

\newtheorem{theorem}{Theorem}

\newtheorem{corollary}{Corollary}

\newcommand{\X}[1]{\ensuremath{\Xi^{[#1]}}}

\title[Hausdorff Dimension of Fibonacci Fractals]{Hausdorff Dimension of Generalized Fibonacci Word Fractals}

\author{Tyler Hoffman and Benjamin Steinhurst}
\address[T. Hoffman and B. Steinhurst]{Department of Mathematics and Computer Science, McDaniel College, Westminster MD 21157}

\date{\today}

\begin{document}

\begin{abstract}
Fibonacci word fractals are a class of fractals that have been studied recently, though the word they are generated from is more widely studied in combinatorics. The Fibonacci word can be used to draw a curve which possesses self-similarities determined by the recursive structure of the word. The Hausdorff dimension of the scaling limit of the finite Fibonacci word curves is computed for $i$-Fibonacci curves and any drawing angle between $0$ and $\pi/2$.
\end{abstract}

\maketitle

\section{Introduction}

Motivated by the work of Monnerot-Dumaine \cite{MonnFrac}, we analyze a generalized class of Fibonacci fractals for their Hausdorff dimension. The two parameters we consider are the drawing angle as introduced by Monnerot-Dumaine and the use of the $i-$Fibonacci words as introduced by Ram\'irez et al \cite{RamirezRubianoMathematica,RamirezBiperiodic,Ramirez2014}. The necessity of this analysis is that the dimensions estimated in the previous works are for the self-similarity dimension despite occasional references to Hausdorff dimension. The equality of these two dimensions often holds but is not automatic. Equality holds in the familiar case when there exists an iterated function system satisfying the open set condition whose attractor is the fractal in question. Many of the already mentioned papers and those of Blondin-Mass\'e \cite{BlondinMasse2011} also include a study of the Fibonacci snowflake fractal which we do not consider in the present paper. 

The basic construction of the Fibonacci fractals is based on three operations: one is a recursive construction of a string of symbols, the second is the processing of a finite string into a curve, and the third is taking a scaling limit of the finite curves to a fractal limit. The construction of the recursively defined words is what gives the name Fibonacci to these fractals. Fibonacci words, as strings of $0's$ and $1's$ have been extensively studied since at least the 1980's. See for a few examples \cite{Berstel1985,Chuan1992}. The processing of the finite words into curves is inspired, though not identical, to a similar process in $L-$systems \cite{MonnFrac}.

Much of the notation will look similar, but mean very different things. Any given finite Fibonacci word is written with notation $f_n^{[i]}$. More on what $i$ and $n$ represent later. A finite Fibonacci curve is denoted as $\mathcal{F}_n^{[i]}$. A Fibonacci number itself is shown as $F_n^{[i]}$. It should be noted the difference in the case of the letter and the font in determining the meaning of a symbol. Additionally we denote the fractal associated to an $i-$Fibonacci word with drawing angle $\alpha$ as $\mathcal{F}^{[i]}_{\alpha}$.

We begin by first stating the definitions of the Fibonacci and i-Fibonacci words along with some of their particularly useful properties (Section \ref{sec:fibwords}). We then draw the corresponding curves and explore their properties as well (Section \ref{sec:fibcurves}). Finally we demonstrate the existence of the fractals and most importantly calculate the Hausdorff dimension of any given i-Fibonacci fractal with drawing angle $0 \le \alpha\le\frac{\pi}{2}$ (Section \ref{sec:fibfrac}). The final theorem of the paper will be the continuity of the fractal construction as a function of the drawing angle, Proposition \ref{prop:angleDimH} and Theorem \ref{thm:FracConverge}.

\subsection*{Acknowledgements}
We thank Zachary Littrell for his aid in reviewing our code and Spencer Hamblen for his reoccurring timely assistance. McDaniel College kindly provided support for this project through a Student-Faculty Collaboration Grant.

\section{Fibonacci Words}\label{sec:fibwords}

The (classical) $n'th$ Fibonacci word is the word $f_n$ defined over the alphabet $\{0,1\}$ that is generated by the concatenation rule $f_n=f_{n-1}f_{n-2}$ with initial values $f_0=1$ and $f_1=0$. Later we will discuss the infinite word $f$. More information on the sequence can be found in the On-line Encyclopedia of Integer Sequences for sequence number A003849 \cite{oeisA003849}. Since the first $F_n$ digits of $f_m$ $m > n$ are fixed it is possible to define $f_{\infty}$ as the infinite Fibonacci word. We however show the existence of $f_{\infty}$ metrically in Proposition \ref{prop:infty}.

The $i-$Fibonacci word is constructed according the the same rules as a classical Fibonacci word however the initial values are $f_1=0$ and $f_2=0^{i-1}1$. We use $0^{i-1}$ to represent a string of $0$'s of length $i-1$.  Table \ref{tab:Fibwords} lists the first five $i-$Fibonacci words for $n=2$, the standard Fibonacci words, and $n=3.$ 

\begin{table}[t]
\caption{}
\begin{center}
\begin{tabular}{|c|c|c|}\hline
n & $f_n^{[2]}$ & $f_n^{[3]}$ \\ \hline
1 &0 & 0 \\ \hline
2 &01 &001\\ \hline
3 &010 &0010\\ \hline
4 &01001&0010001\\ \hline
5 &01001010&00100010010\\ \hline
\end{tabular}
\end{center}
\label{tab:Fibwords}
\end{table}%

While concatenation is the basic construction technique for Fibonacci words we will define a substitution rule that will produce $f_n^{[i]}$ from $f_1^{[i]}$ as well. The main utility is that the application of a substitution rule is a local operation while concatenation is global operation. This perspective will be more useful in the discussion of the curves in the next section. It should be noted that while our $\X{i}$ substitution is similar to the $\sigma$ substitution used in \cite{MonnFrac}, it is not the same. 

\begin{definition}
Define \X{i} to be the substitution rule: 
$$\X{i}: \left\{\begin{array}{l} 0\to 0^{i-1}1 \\ 0^{i-1}1\to 0^{i-1}10 \end{array} \right.$$
over words comprised of blocks of the form $0$ and $0^{i-1}1$.
\end{definition}

Since $\X{i}$ transforms a sequence of blocks $0$ and $0^{i-1}1$ into another sequence of blocks of $0$ and $0^{i-1}1$ the transformation can be iterated. 

\begin{proposition}
Let $f^{[i]}_{n}$ be an $i-$Fibonacci word then $\X{i}(f^{[i]}_n) = f^{[i]}_{n+1}$.
\end{proposition}

Since $f^{[i]}_1$ is unambiguous in how it is written as a single block of $0$ we will always assume that $f^{[i]}_n$ is broken down into these blocks in the way that it arises as the output of this process. 

\begin{proof}
We proceed by induction. Note that $f^{[i]}_1 = 0$ and $\X{i}(0) = 0^{i-1}1 = f^{[i]}_2$. Similarly that $\X{i}(f^{[i]}_1) = f^{[i]}_2$ can be checked manually. Suppose the proposition holds for $n-1$ and $n-2$ and we will show it holds for $n$. Since $\X{i}$ acts locally by replacing blocks and $f_n = f_{n-1}f_{n-2}$ can be split into two sub-collections of blocks the action of $\X{i}$ can be split over $f_{n-1}$ and $f_{n-2}$ so that 
$$\X{i}(f^{[i]}_n) = \X{i}(f^{[i]}_{n-1}f^{[i]}_{n-2}) = \X{i}(f^{[i]}_{n-1})\X{i}(f^{[i]}_{n-2}) = f^{[i]}_{n}f^{[i]}_{n-1} = f^{[i]}_{n+1}.$$
\end{proof}

The following properties of the $i-$Fibonacci words will be useful later. They are cited here without proof.

\begin{proposition}\label{prop:wordprops}
\cite[Proposition 5]{Ramirez2014} The following properties hold for all i-Fibonacci words $f_n^{[i]}$.
\begin{enumerate}
\item The subword "11"  can never be found in an i-Fibonacci word with $i\ge 2$.
\item The word $f_n^{[i]}$ can be decomposed as $p_nab$ where $p_n$ is a palindrome and $ab$ depend on the parity of $n$ only.
\item For even $n$, $ab=10$ and for odd $n$, $ab=01$
\item The i-Fibonacci word has a five-partite structure 
$$f_n^{[i]}=f_{n-3}^{[i]}f_{n-3}^{[i]}f_{n-6}^{[i]}l_{n-3}^{[i]}l_{n-3}^{[i]},$$
where $l_{n}^{[i]}=p_n^{[i]}ba$ (i.e. the last two letters are swapped).
\end{enumerate}
\end{proposition}

\begin{proposition}\label{prop:infty}
The infinite i-Fibonacci words exist for any given i $ \ge 2$ and are $f^{[i]}=\lim_{n\to\infty}f_n^{[i]}$.
\end{proposition}

\begin{proof}
The limit is taken in the sense of the standard $2-$adic metric on the space of finite and infinite words. That is $d_2(w,v) = 2^{-n}$ where $n$ is the number of initial symbols that $w$ and $v$ share. By the definition of $f_n^{[i]}$ as concatenations, for each $i$ the sequence is Cauchy in this metric. Thus there is a limit word that is denoted $f^{[i]}$
\end{proof}

It was mentioned earlier that the infinite Fibonacci words could also be considered as fixed points of the concatenation rule. This metric proof of the existence is equivalent to that argument by the definition of the metric measuring distance between two words by measuring the length of the longest shared prefix of those two words. 

\section{Fibonacci Curves}\label{sec:fibcurves}

The Fibonacci word curve $\mathcal{F}^{[i]}_n$ is generated by taking the corresponding Fibonacci word $f^{[i]}_n$ and applying the following drawing rule. The Fibonacci curves introduced in \cite{MonnFrac} are just the case when $i=2$. 

{\bf Drawing Rule:}
 Let $a_j$ be the $j$'th character of $f_n$ and perform the following procedure over each character $a_j$.
\begin{enumerate}
\item Set initial direction, $a(\emptyset)=\frac{\pi}{2}$
\item Choose drawing angle $\alpha$.
\item Draw a segment in the current direction.
\item If $a_j$ is 0 then
\begin{itemize}
	\item If $j$ is even, turn left, i.e. add $\alpha$ to the current direction. 
    \item If $j$ is odd, turn right, i.e. subtract $\alpha$ from the current direction.
\end{itemize}
\item If $a_j$ is 1, do not change direction.
\item Repeat from step 3 to step 5 for $j=1$ to $j = F^{[i]}_n.$
\end{enumerate}

It should be noted that the instructions ``turn right'' or ``turn left'' mean to change the angle at which the next segment will be drawn. When the end of a word is reached by the drawing rule, there will be an angle pointing to where the next segment would be drawn and this final direction will be called the net angle of a word.

\begin{definition}
Let $a(w)$ be a function on finite words on the alphabet $\{0,1\}$. Set $a(\emptyset) = \pi/2$, this value is the ``initial angle.'' Then $a(wb) = a(w)+\psi$ where 
\begin{itemize}
	\item $\psi=0$ if $b=1$,
	\item $\psi=1$ if $b=0$ and $|wb|$ is even,
	\item $\psi=1$ if $b=0$ and $|wb|$ is odd.
\end{itemize}
\end{definition}

Following from the properties of the i-Fibonacci words are the corresponding properties for the curves. 

\begin{proposition}\label{prop:iCurveProps}
\cite[Proposition 6]{Ramirez2014} The following properties hold for all curves $\mathcal{F}_n^{[i]}$.
\begin{enumerate}
\item There exist only segments of length 1 or 2 in the curve. 
\item The curve $\mathcal{F}_n^{[i]}$ is similar to the curve $\mathcal{F}_{n-3}^{[i]}$ with the same shape, but different number of segments.
\item The scaling factor between curves $\mathcal{F}_n^{[i]}$ and $\mathcal{F}_{n-3}^{[i]}$ is $1+\sqrt{2}$.
\item Similar to the above word property, the curve $\mathcal{F}_n^{[i]}$ has a five-partite structure written
$$\mathcal{F}_n^{[i]}=\mathcal{F}_{n-3}^{[i]}\mathcal{F}_{n-3}^{[i]}\mathcal{F}_{n-6}^{[i]}\mathcal{F'}_{n-3}^{[i]}\mathcal{F'}_{n-3}^{[i]}$$
where $\mathcal{F'}_n^{[i]}$ is the curve corresponding to the word $l_n^{[i]}$.
\end{enumerate}
\end{proposition}

It is assumed that the lengths of the line segments is ``unit length.'' In Section \ref{sec:fibfrac}  scaling ratios between $\mathcal{F}_n^{[i]}$ and $\mathcal{F}_{n-3}^{[i]}$ will be determined. They will depend on $i$, $n$, and also $\alpha$. The scaling ratios will then allow us to take a scaling limit to obtain a Fibonacci fractal. For the time being though, the specific length of the line segments is left unspecified.

Next we want a function that determines what the end behavior of any given subword is in terms of its angle. 

\begin{proposition}\label{prop:6angle}
Let $w$ be a word composed of the blocks $0$ and $0^{i-1}1$. Then $a(w) = - a({\X{i}}^{\circ 3}(w)).$ Furthermore, $a(w) = a({\X{i}}^{\circ 6}(w))$.
\end{proposition}

\begin{proof}
Consider the third substitution on each of two basic blocks: 
\begin{align*}
	(\X{i})^3(0)&=0^{i-1}100^{i-1}1\\ 
    (\X{i})^3(0^{i-1}1) &=0^{i-1}100^{i-1}10^{i-1}10.
\end{align*}
Thus we can see by processing the words according to the drawing rule that $a({\X{i}}^{\circ 3}(0) = -a(0)$. Similarly $a({\X{i}}^{\circ 3}(0^{i-1}1)) = -a(0^{i-1}1)$. By the additivity of the angle function over concatenation of words when \X{i} is applied another three times to the sign of the overall angle is again negated. Thus $a({\X{i}}^{\circ 6}(0)) = a(0)$ and the same for $0^{i-1}1$ as well. 
\end{proof}

From Table \ref{fig:FibCurvesIN} it is clear that as $n$ increases the curve $\mathcal{F}_n^{[i]}$ pass through several global shapes. It is clear that their local structure is similar but for the global structure to be the same and for the orientation of the curve to be the same again it is necessary to consider only when $n\ \mod 6 = k$. Notice that Proposition \ref{prop:6angle} gives another reason for considering every sixth curve as well. For purely historical reasons from \cite{MonnFrac} we consider $k=5$ for even $i$. For odd $i$ we consider $k=3$ as it is visually the most similar. Equivalent analyses could be conducted for any choice of $k$.  The following proposition is the first that makes reference to this choice. 

\begin{proposition}\label{prop:curveAngle}
When $i$ is even 
$a\left(f_{6k+5}^{[i]}\right)=-\alpha$ and when $i$ is odd $a\left(f_{6k+3}^{[i]}\right)=0$. 
\end{proposition}

\begin{proof}
By Proposition \ref{prop:6angle} $a(f^{[i]}_n) = a(f^{[i]}_{n+6})$. So $a\left(f_{6k+5}^{[i]}\right)=a(f^{[i]}_5) = -\alpha$ for even $i$ and $a\left(f_{6k+3}^{[i]}\right)=a(f^{[i]}_3) = 0$ for odd $i$.
\end{proof}

Finally, we verify that there exists no overlap between any two blocks of the five-partite structure for a later proof of Proposition \ref{prop:OSCsatisfied}.

\begin{definition}
The bounding box of a given curve $\mathcal{F}_n^{[i]}$ is the smallest rectangle enclosing the curve. 
\end{definition}

\begin{example}
Refer to Figure \ref{fig:BoundBox} depicting $\mathcal{F}_{11}^{[2]}$ where the dotted red rectangle is the bounding box formed by vertices A, B, C, and D.

\begin{figure}[t]
\begin{center}
\includegraphics[width=3in]{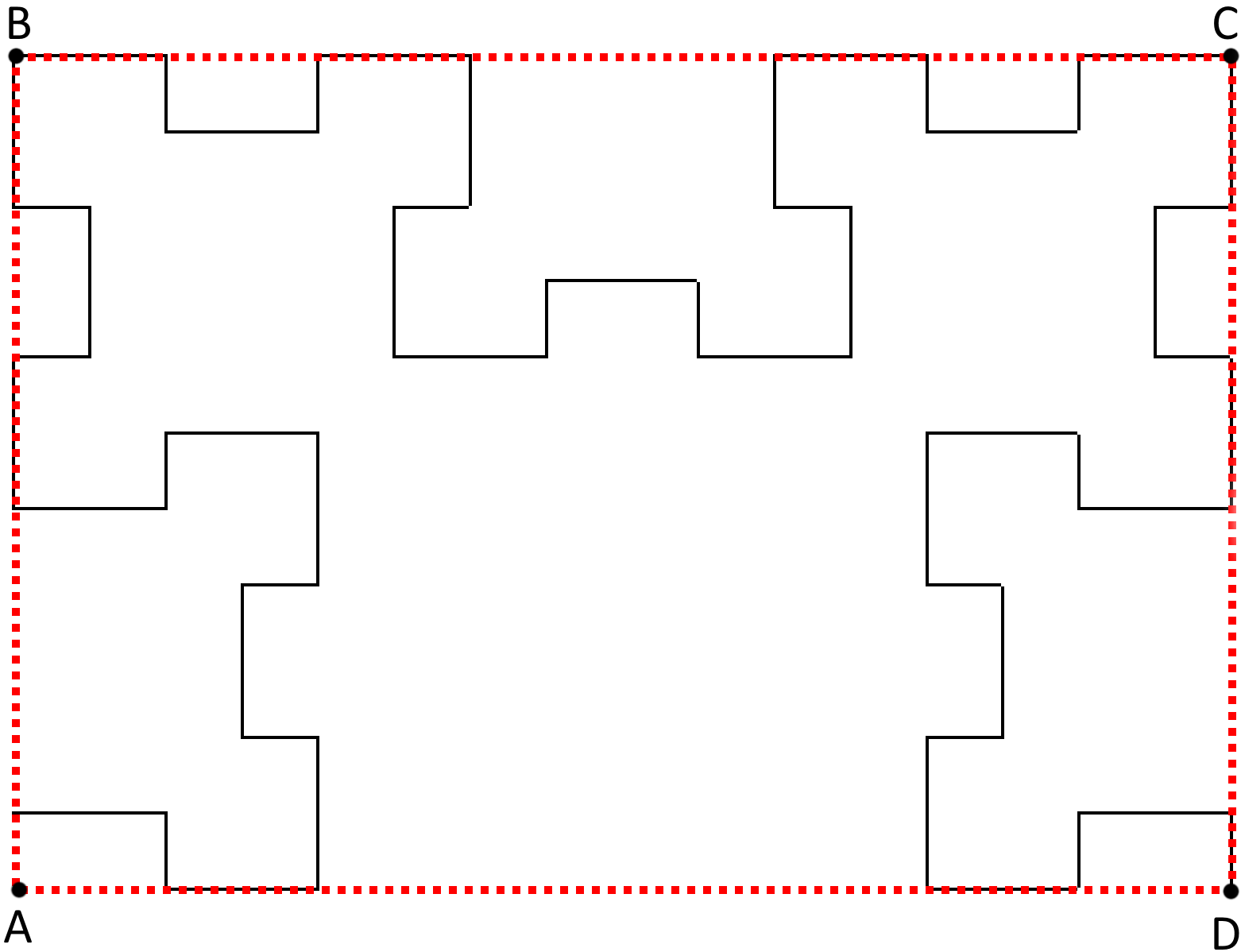}
\caption{An example of a bounding box for $\mathcal{F}^{[2]}_{11}$, $\alpha = \frac{\pi}{2}$. Notice that the initial line segment begins at vertex A and that the last line segment is at vertex D.}
\label{fig:BoundBox}
\end{center}
\end{figure}

\end{example}

\begin{proposition}\label{prop:boundForm}
Consider for even $i$ $\mathcal{F}^{[i]}_{6k+5}$ and for odd $i$ $\mathcal{F}_{6k+3}^{[i]}$. Then the first and last vertices of these curves are co-located with two adjacent vertices of the bounding box.
\end{proposition}

\begin{proof}
Consider the even $i$ case. Curves of the form $\mathcal{F}_{6k+5}^{[i]}$ have first and last vertices on a line in $\mathbb{R}^2$ and we wish to show that no other segments cross this line. It is true for $\mathcal{F}_5^{[i]}$ because for any even $i$ it consists of a diagonal walk up, a double segment, and a diagonal walk down of equal length providing a net right turn. This gives that the net change in the vertical coordinate is zero. Notice that, with this description of $\mathcal{F}_5^{[i]}$, it is entirely contained inside of its bounding box. It is true for $\mathcal{F}_8^{[i]}$ because it is the assemblage of four copies of $\mathcal{F}_5^{[i]}$ and an $\mathcal{F}_2^{[i]}$ in the shape shown in Figure \ref{fig:PhiBox}.

Recall that each curve $\mathcal{F}_{6k+5}^{[i]}$ is made up of curves $\mathcal{F}_{6k+2}^{[i]}$ and $\mathcal{F}_{6k-1}^{[i]}$ (from the five-partite structure). Assuming that the claim is true for $\mathcal{F}_{6k+2}^{[i]}$ and $\mathcal{F}_{6k-1}^{[i]}$, then it holds for their assemblage because of their arrangement in the five-partite structure.

Finally, since the case for odd $i$ introduces only a rotation of $\frac{\pi}{4}$ and rotation does not change these properties, it holds for this case as well.
\end{proof}

\begin{figure}[t]
\begin{center}
\includegraphics[width=3in]{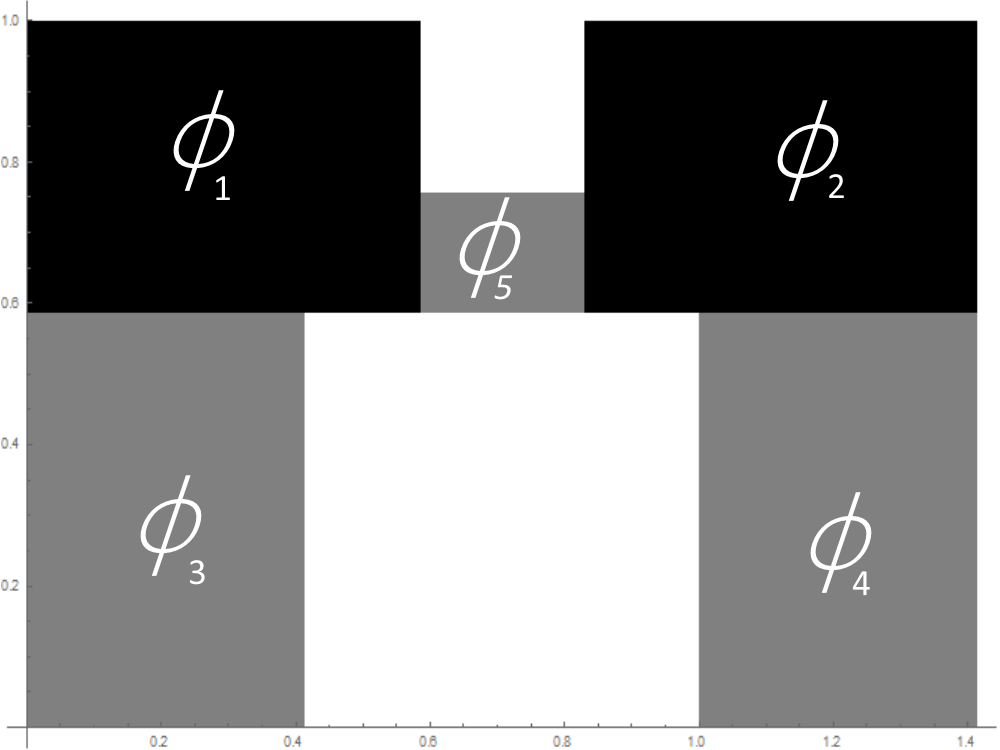}
\caption{The mappings $\phi_k$ when $\alpha = \frac{\pi}{2}$ from the IFS over the set $V$, the box they are contained in.}
\label{fig:PhiBox}
\end{center}
\end{figure}

\begin{proposition}\label{prop:boundOverlap}
The bounding boxes of each individual sub-curve of $\mathcal{F}_n^{[i]}$ from the five-partite structure have disjoint interiors.
\end{proposition}

\begin{proof}
From the five-partite structure, we know that the last vertex of one sub-curve is the first vertex of the next sub-curve since the sub-curves are connected and form a continuous path. From Proposition \ref{prop:boundForm}, we know that the first and last vertices of a sub-curve are also vertices of the bounding box. This means that the bounding boxes of the two sub-curves that share a vertex also share that same vertex. Due to the fact that we draw with an $\alpha\le\frac{\pi}{2}$ angle and that this vertex is specifically the first or last vertex, no segment from one bounding box, that is one sub-curve, will ever intersect with the interior of the adjacent sub-curve's bounding box.

It remains to be shown that non-adjacent sub-curves also have bounding boxes with disjoint interiors. Specifically, we need to show that that the bounding boxes for the first and last sub-curves do not intersect. Because of the five-partite structure, we know that the first and last boxes are the same as the second and fourth boxes, but the second and fourth are at a rotation of $\alpha$. From Proposition \ref{prop:AspectRatio}, we know that these boxes are taller than they are wide, the distance created by the middle three boxes is greater than the combined width of the first and last boxes, thus, causing no overlap.
\end{proof}

The first visual impressions from Figure \ref{fig:FibCurvesIN} is that as $i$ changes from even to odd and back again that to pick a particular ``Fibonacci fractal'' is not a canonical choice. We follow \cite{MonnFrac} and choose for $i$ even $n=6k+5$  as the approximating curves to a fractal. For $i$ odd we choose $n=6k+3$ because it is only varies from the first choice by a rotation. 

\begin{figure}[t]
\begin{center}
\begin{tabular}{|c||c|c|c|c|}\hline
 & $n=15$ & $n=16$ & $n=17$ & $n=18$ \\ \hline \hline
$i=2$ & \includegraphics[width=.8in]{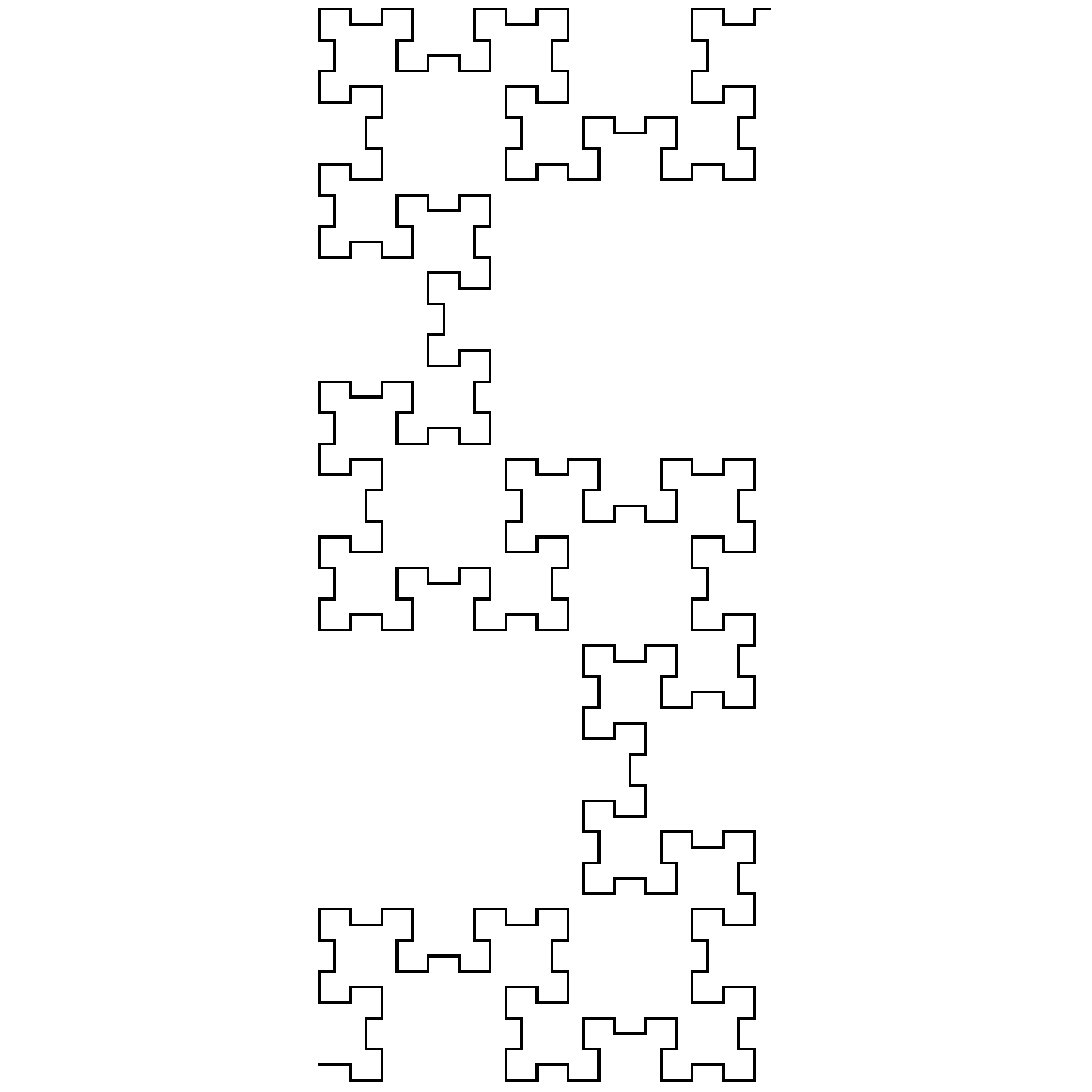} & \includegraphics[width=.8in]{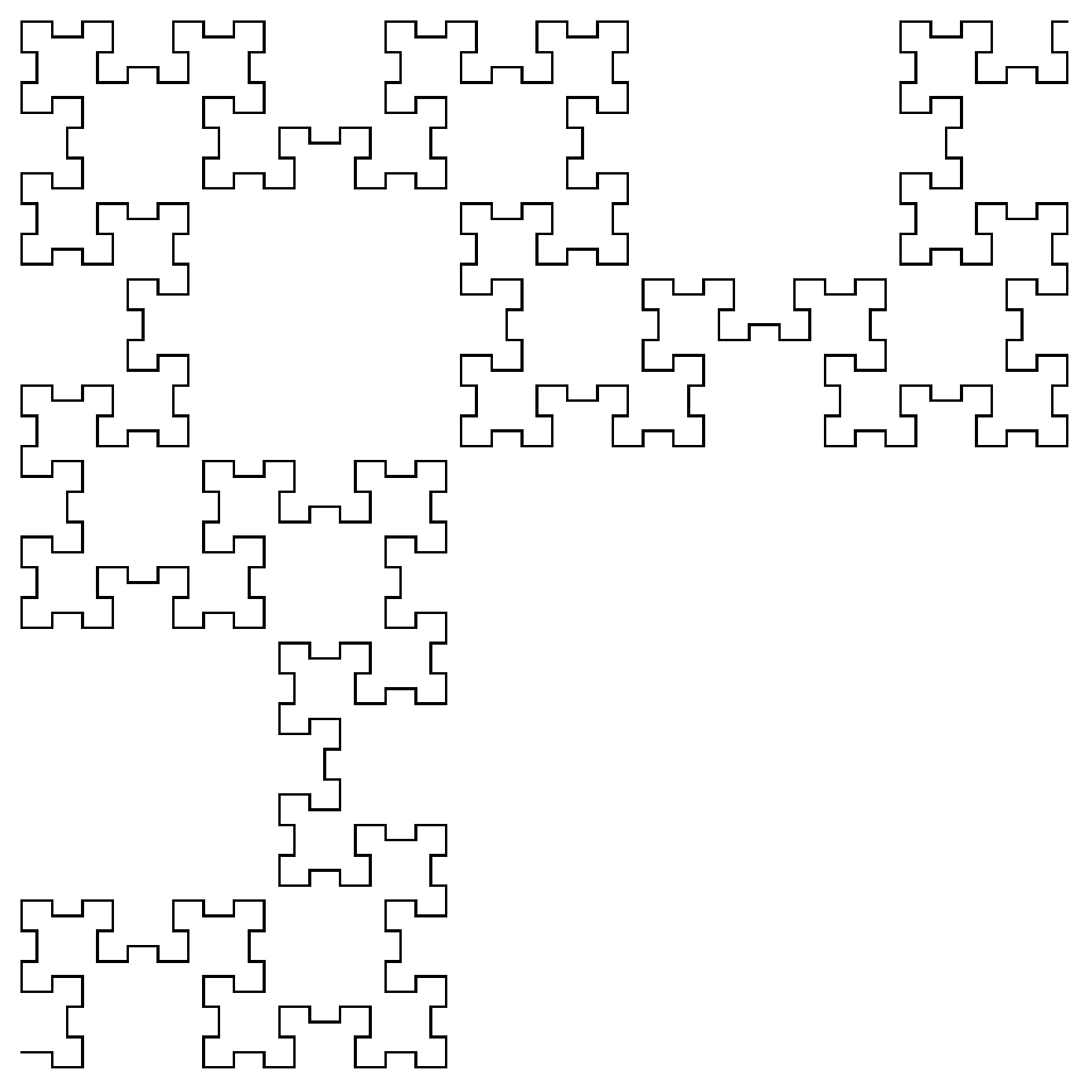} & \includegraphics[width=.8in]{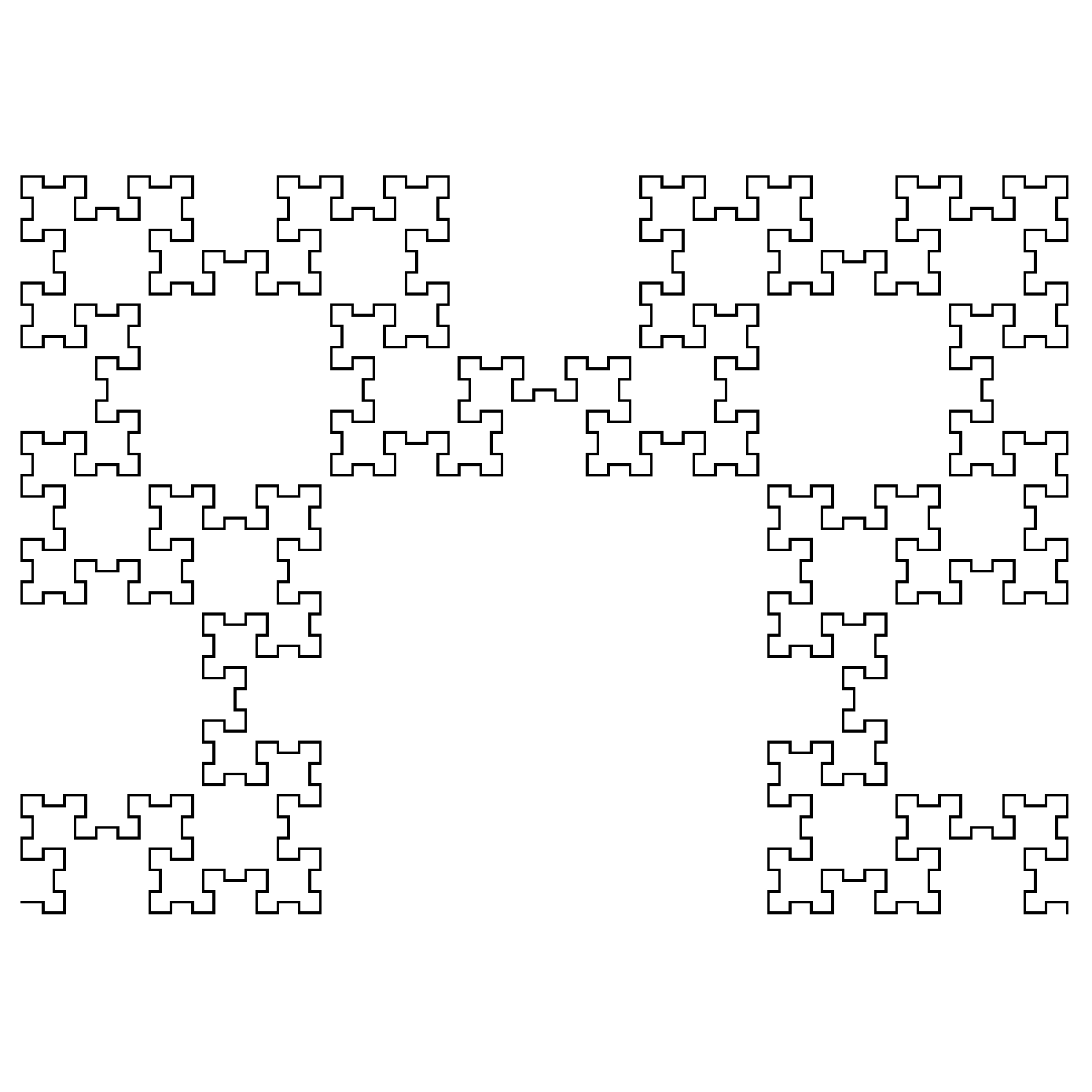} & \includegraphics[width=.8in]{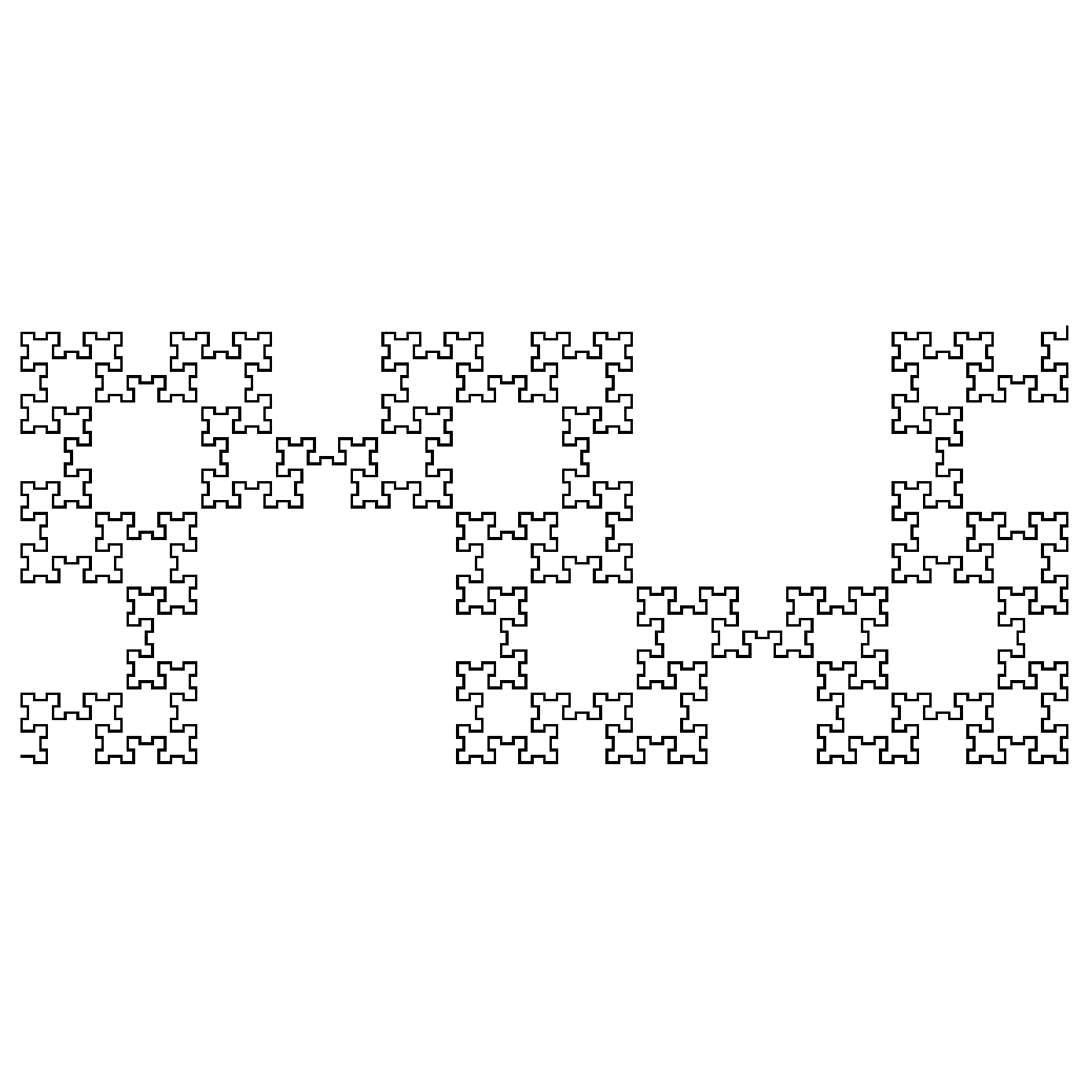} \\ \hline

$i=3$ & \includegraphics[width=.8in]{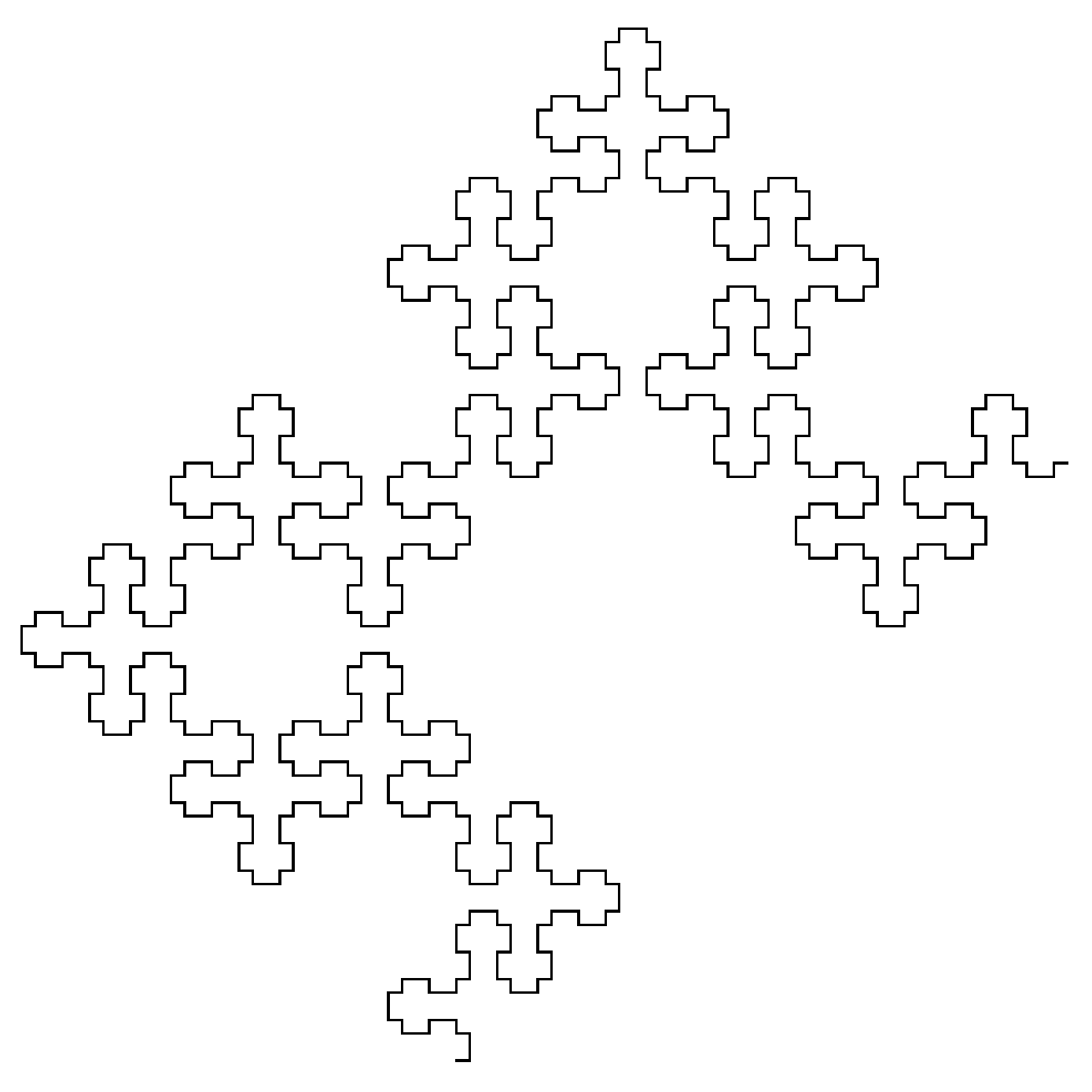} & \includegraphics[width=.8in]{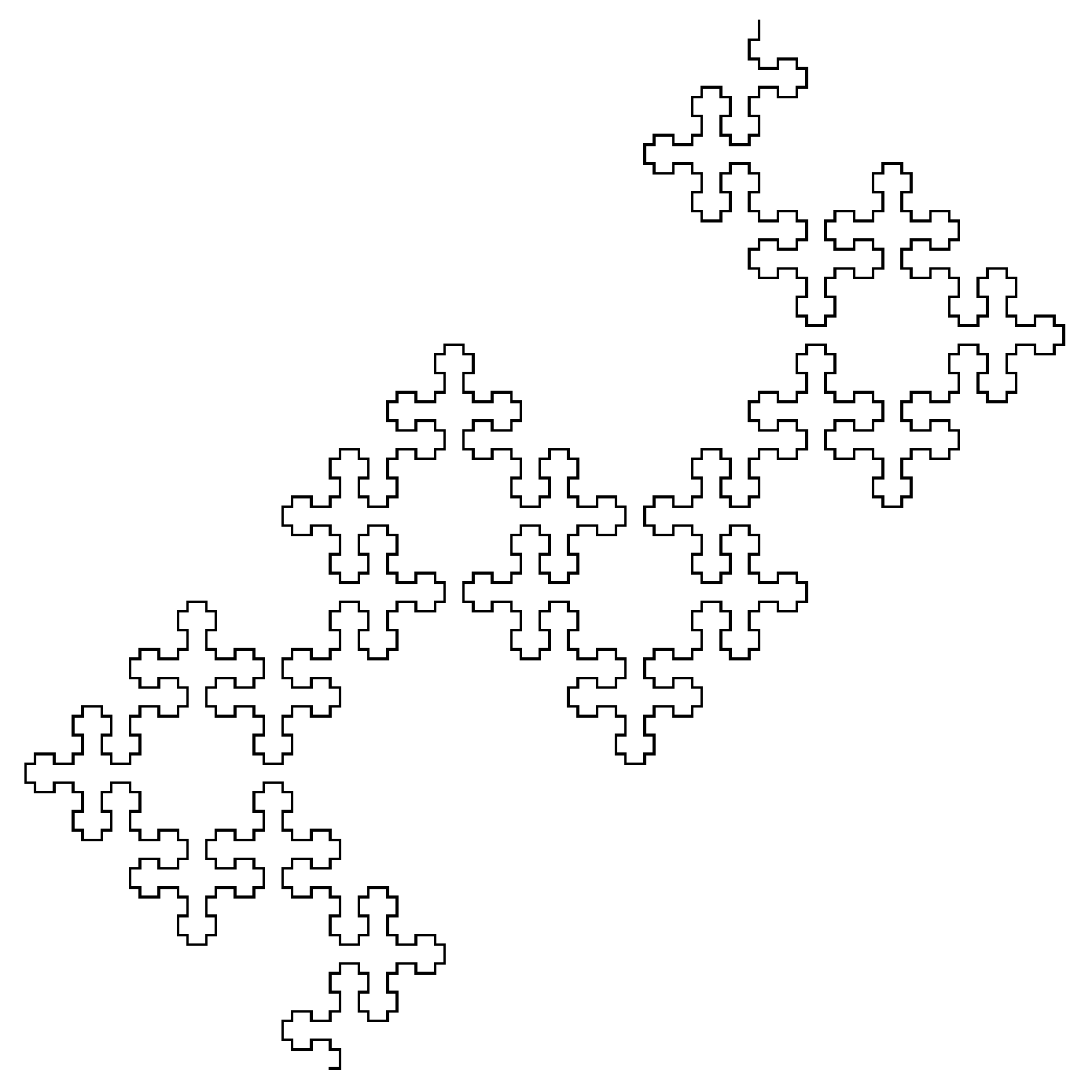} & \includegraphics[width=.8in]{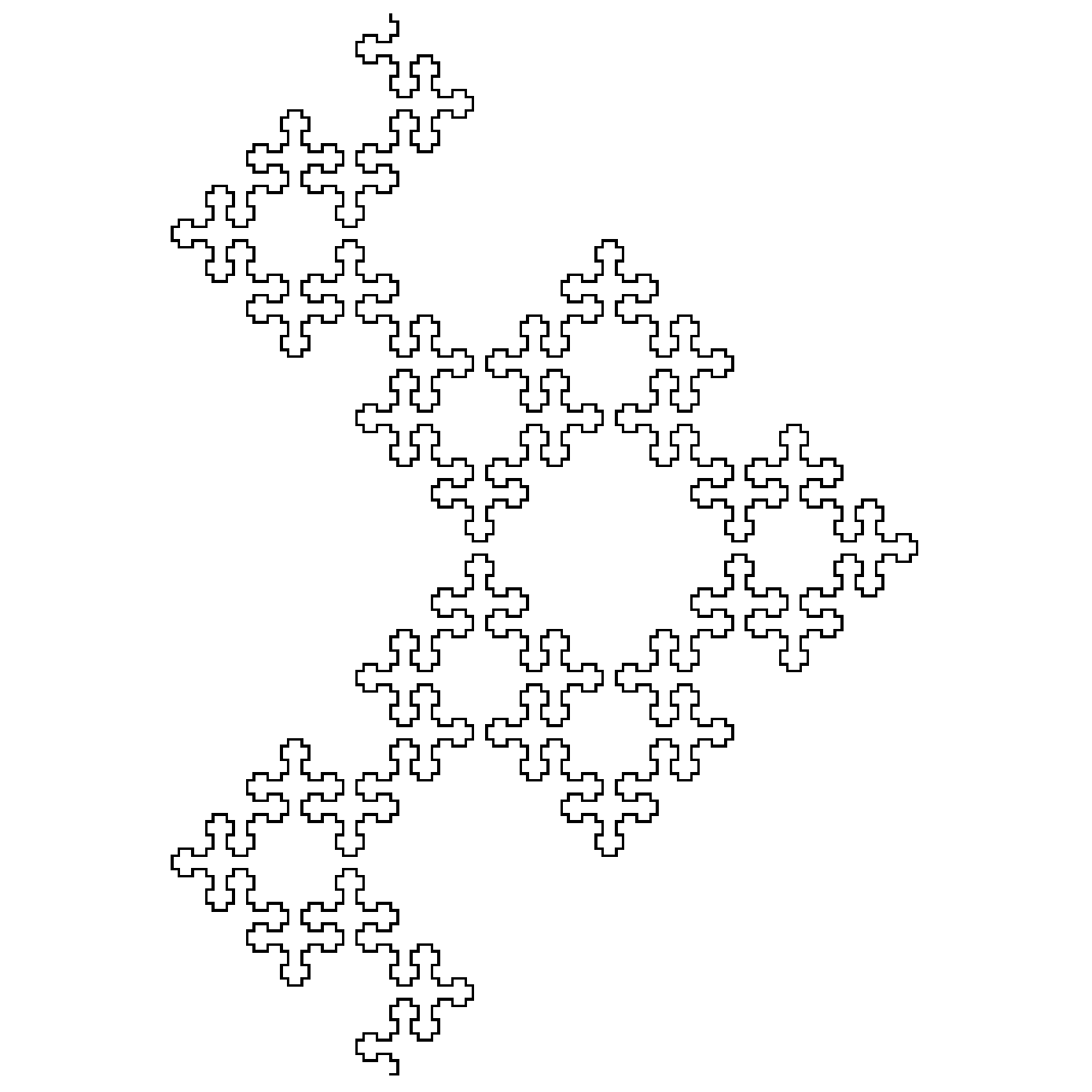} & \includegraphics[width=.8in]{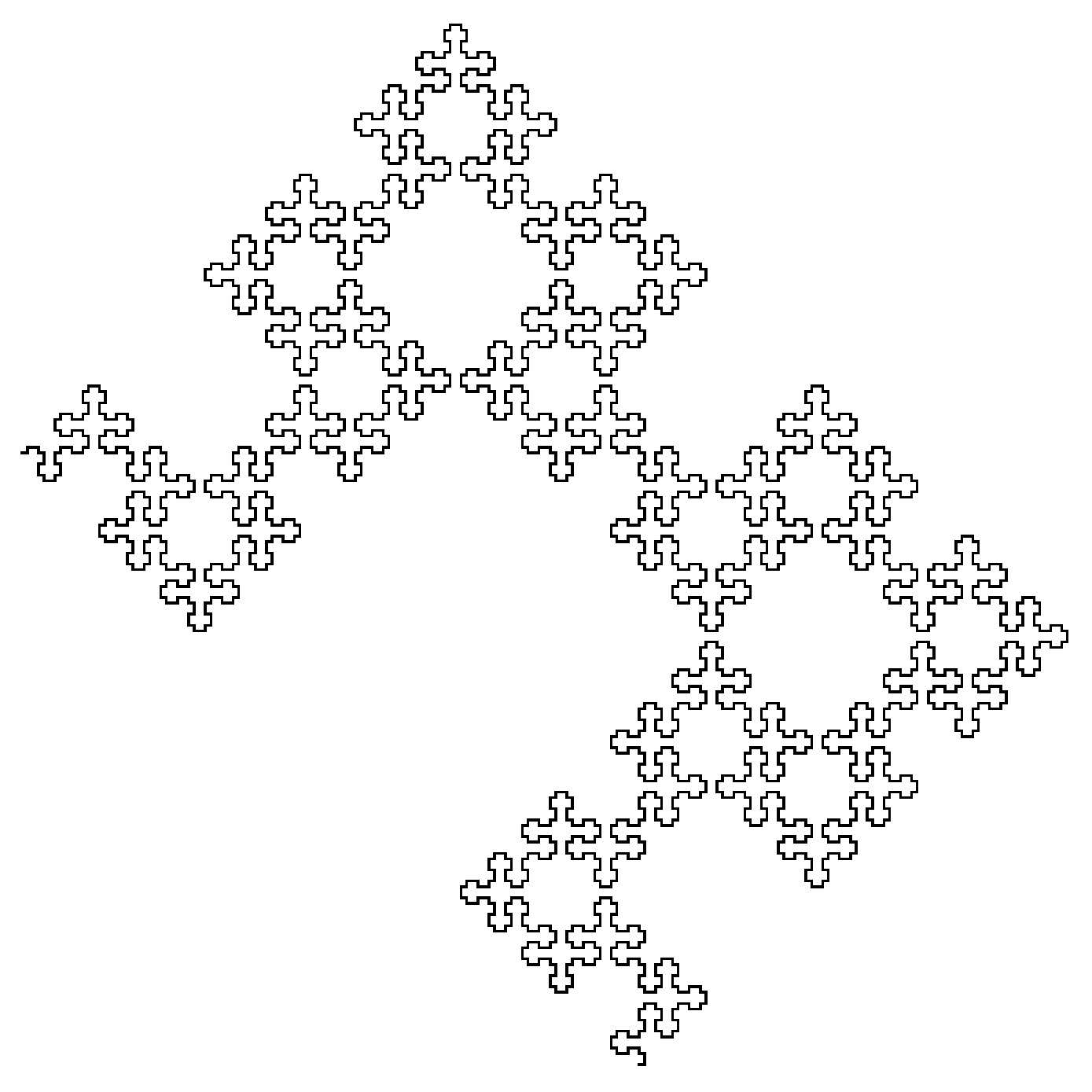} \\ \hline

$i=4$ & \includegraphics[width=.8in]{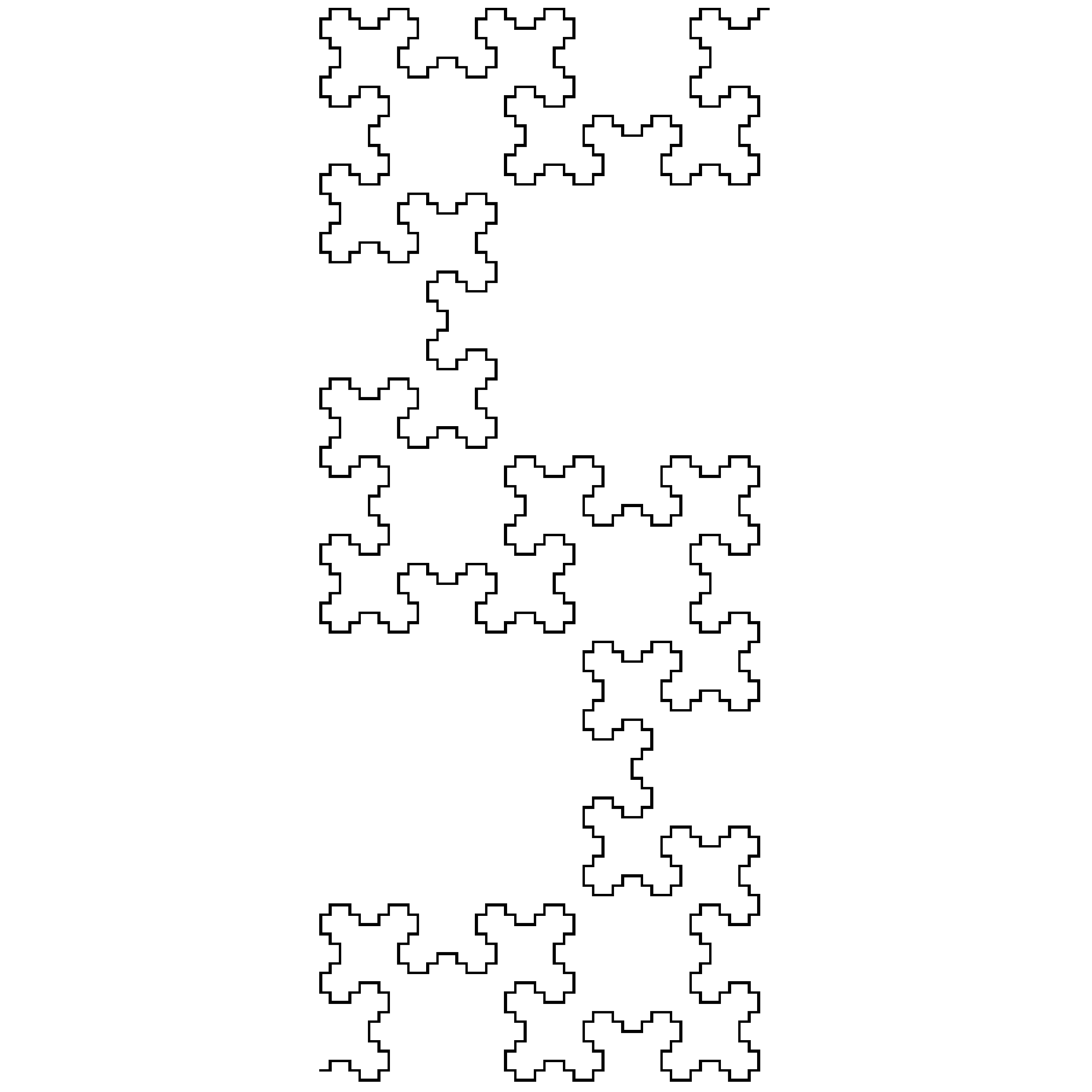} & \includegraphics[width=.8in]{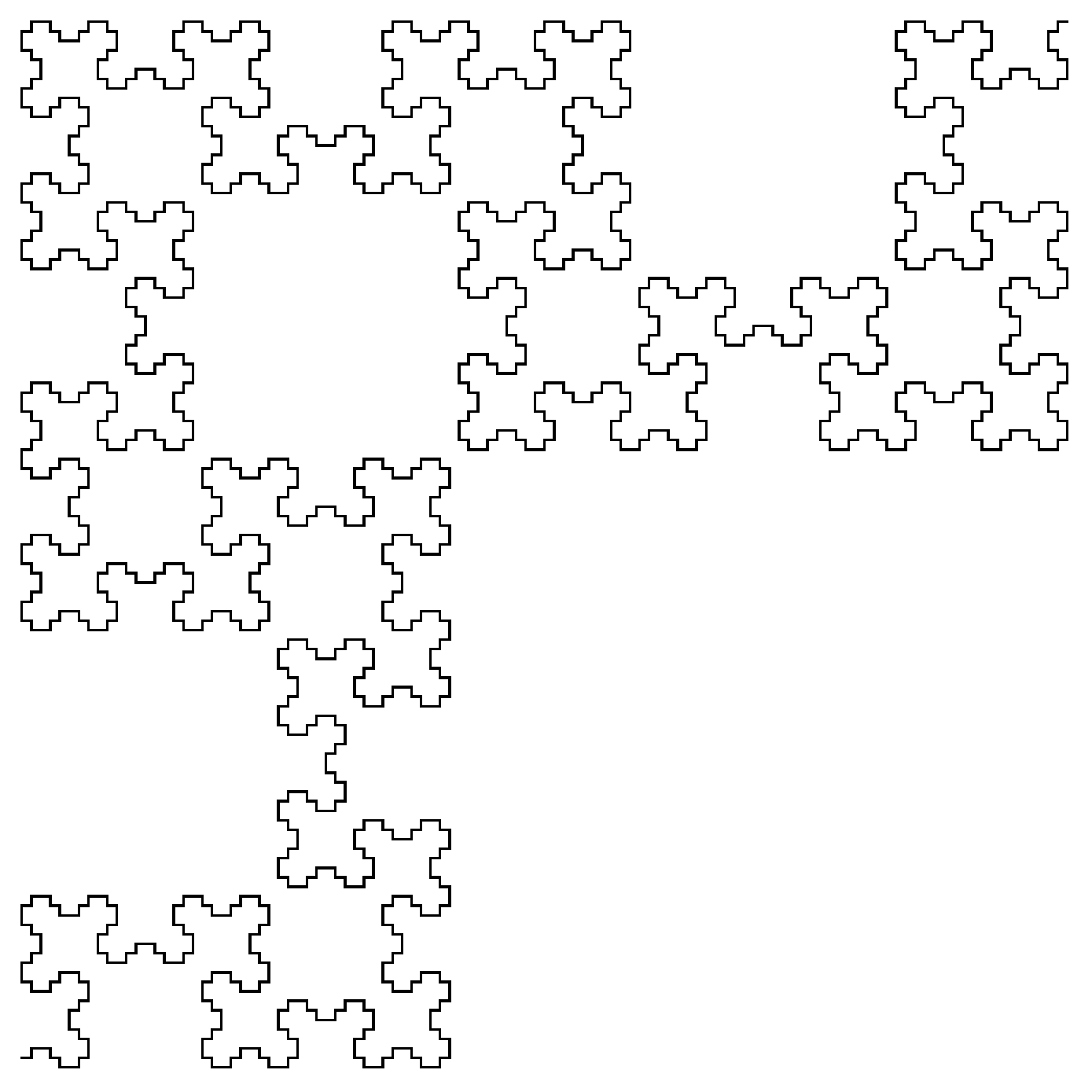} & \includegraphics[width=.8in]{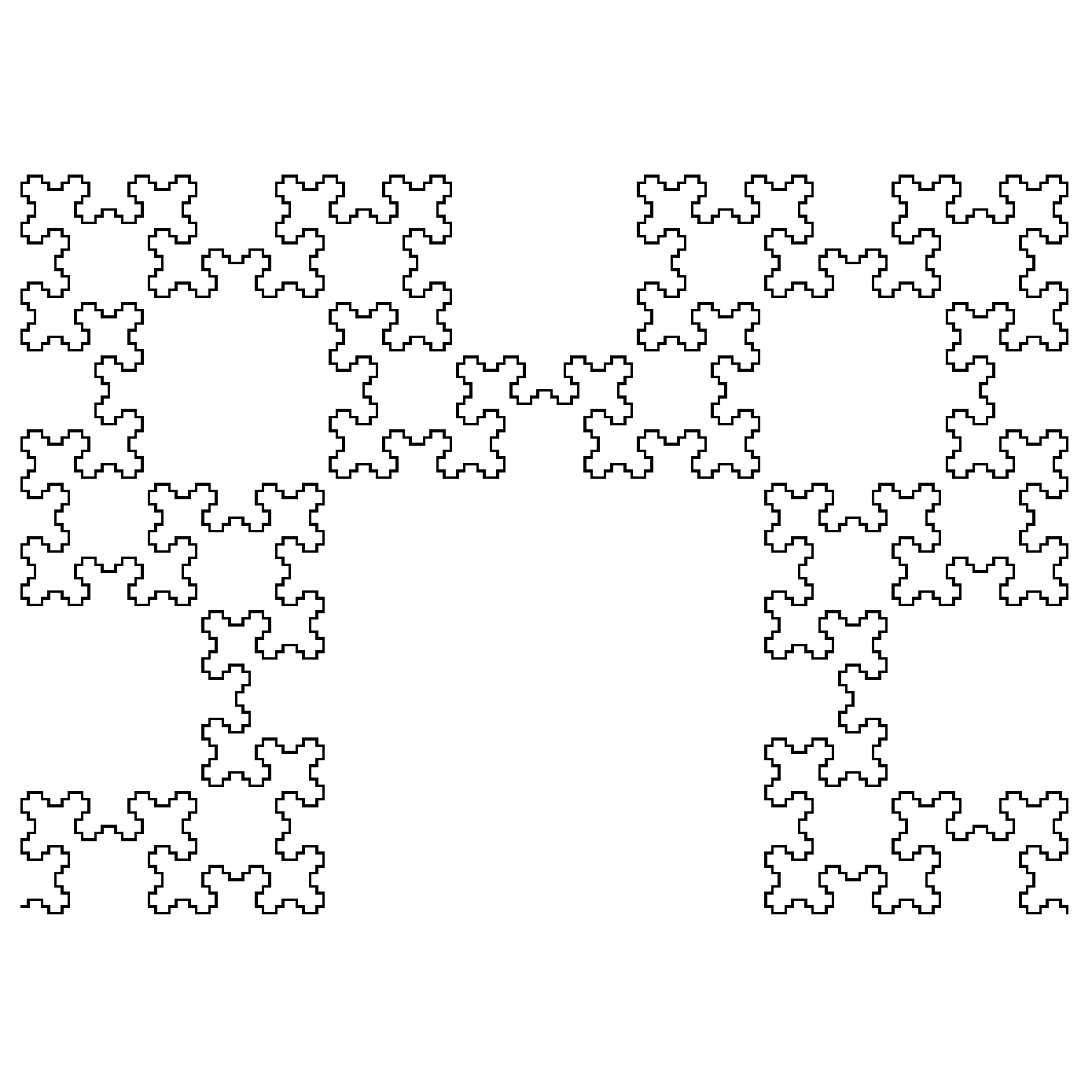} & \includegraphics[width=.8in]{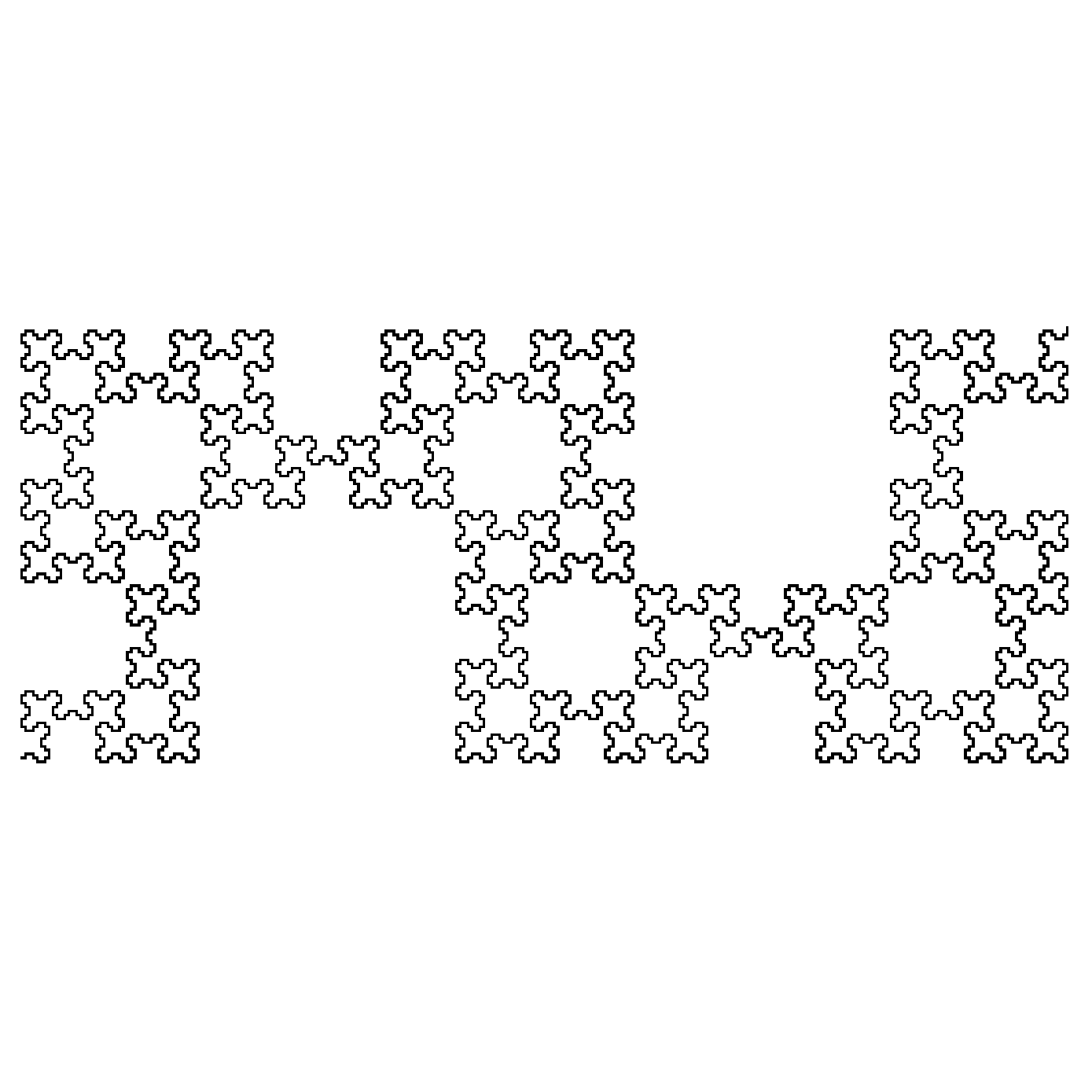} \\ \hline

$i=5$ & \includegraphics[width=.8in]{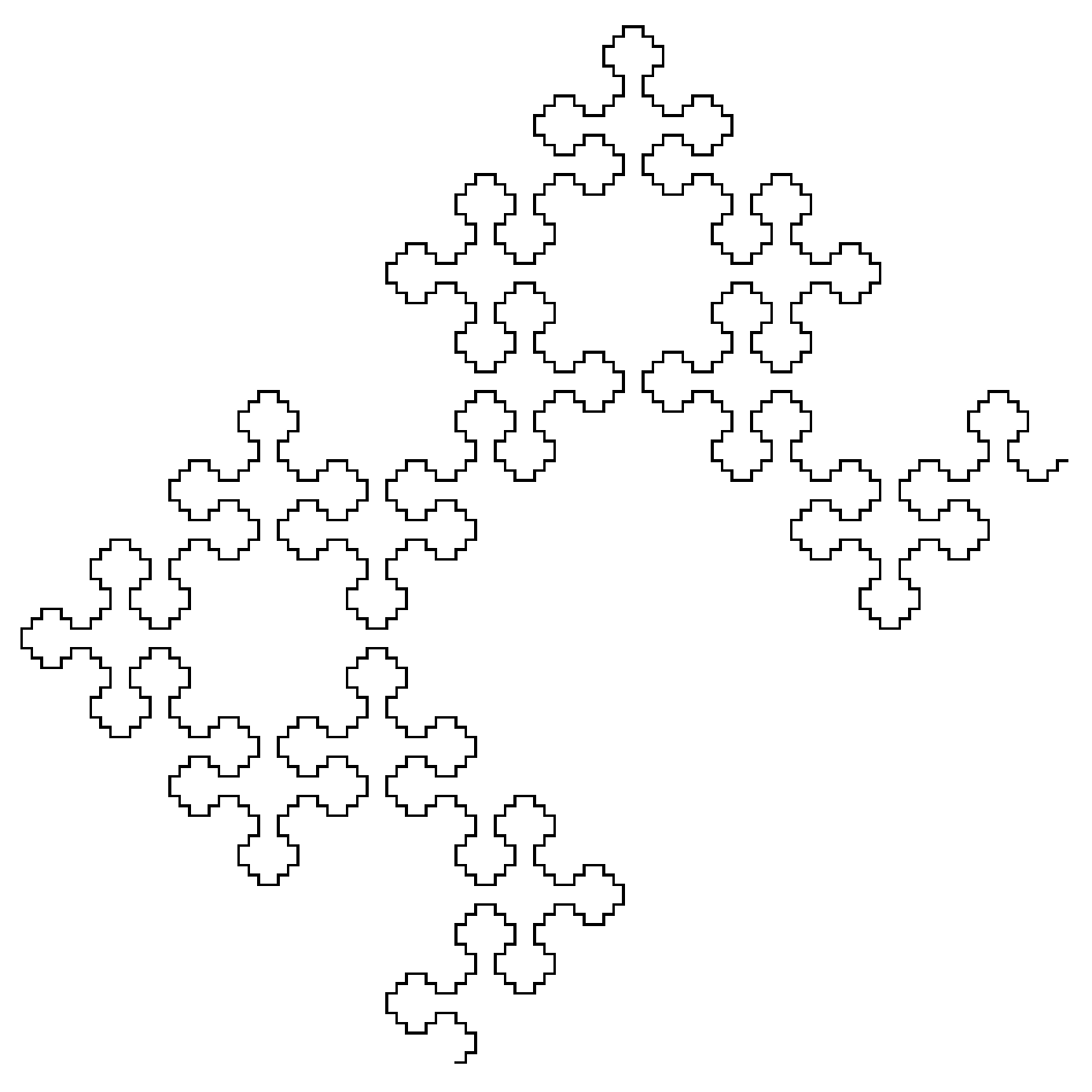} & \includegraphics[width=.8in]{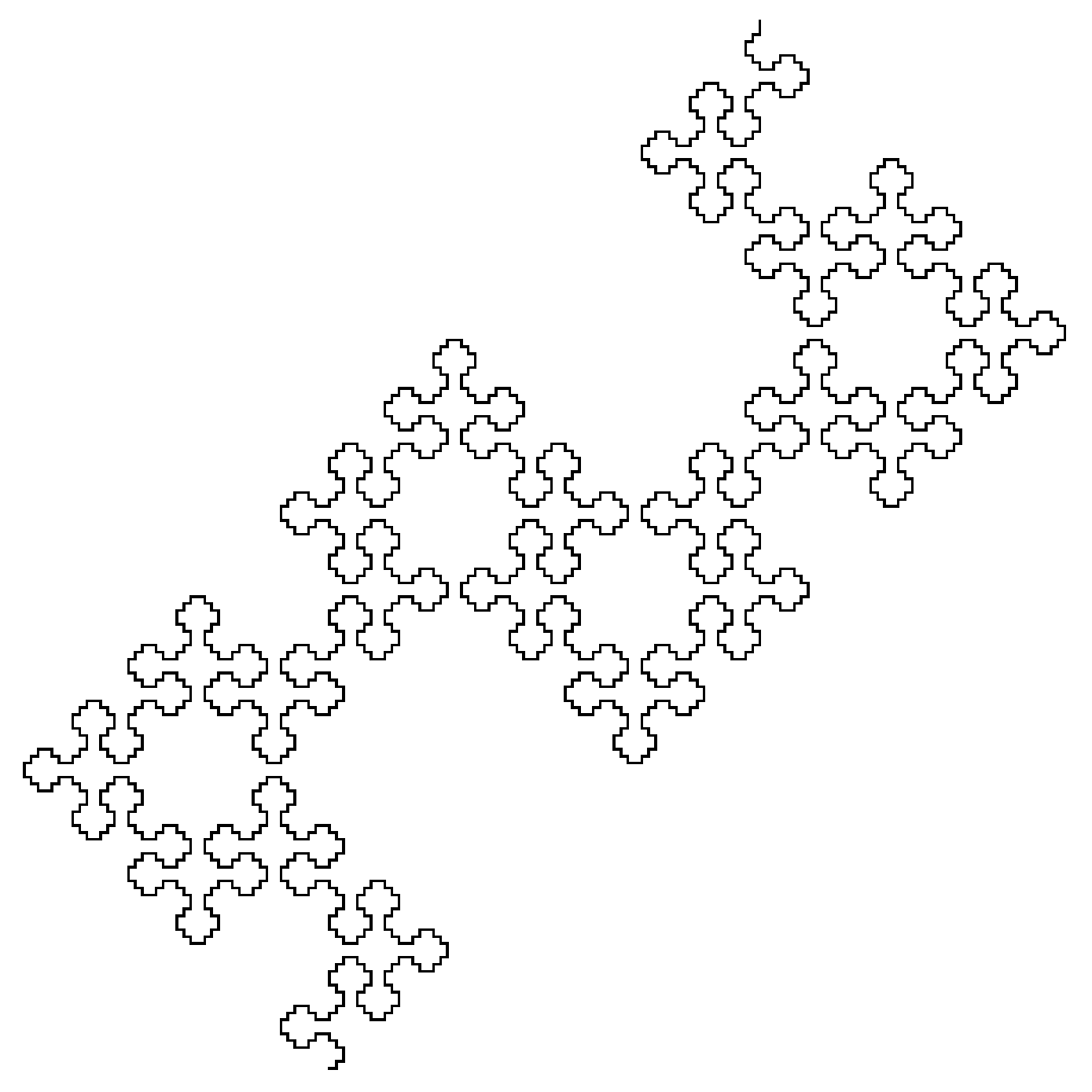} & \includegraphics[width=.8in]{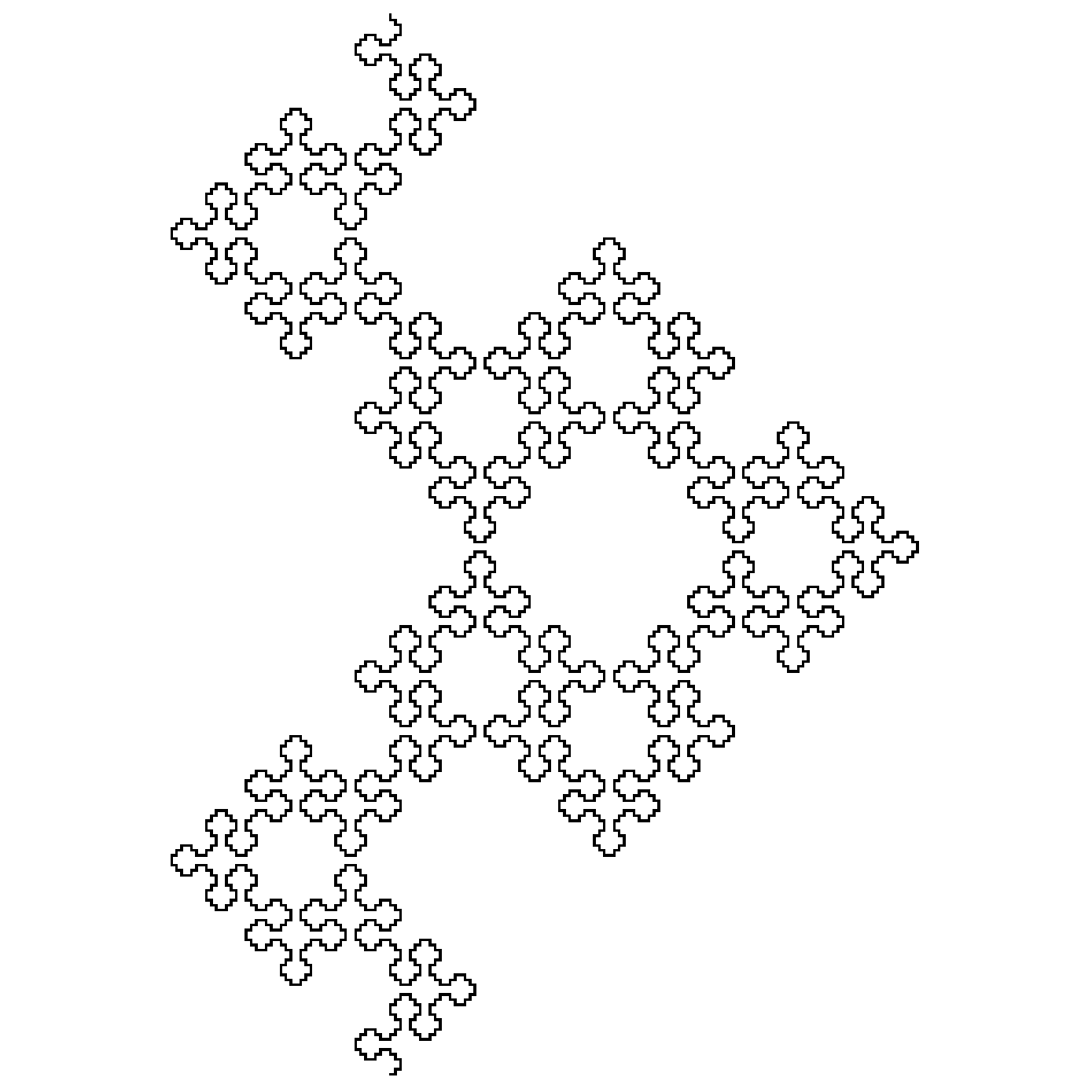} & \includegraphics[width=.8in]{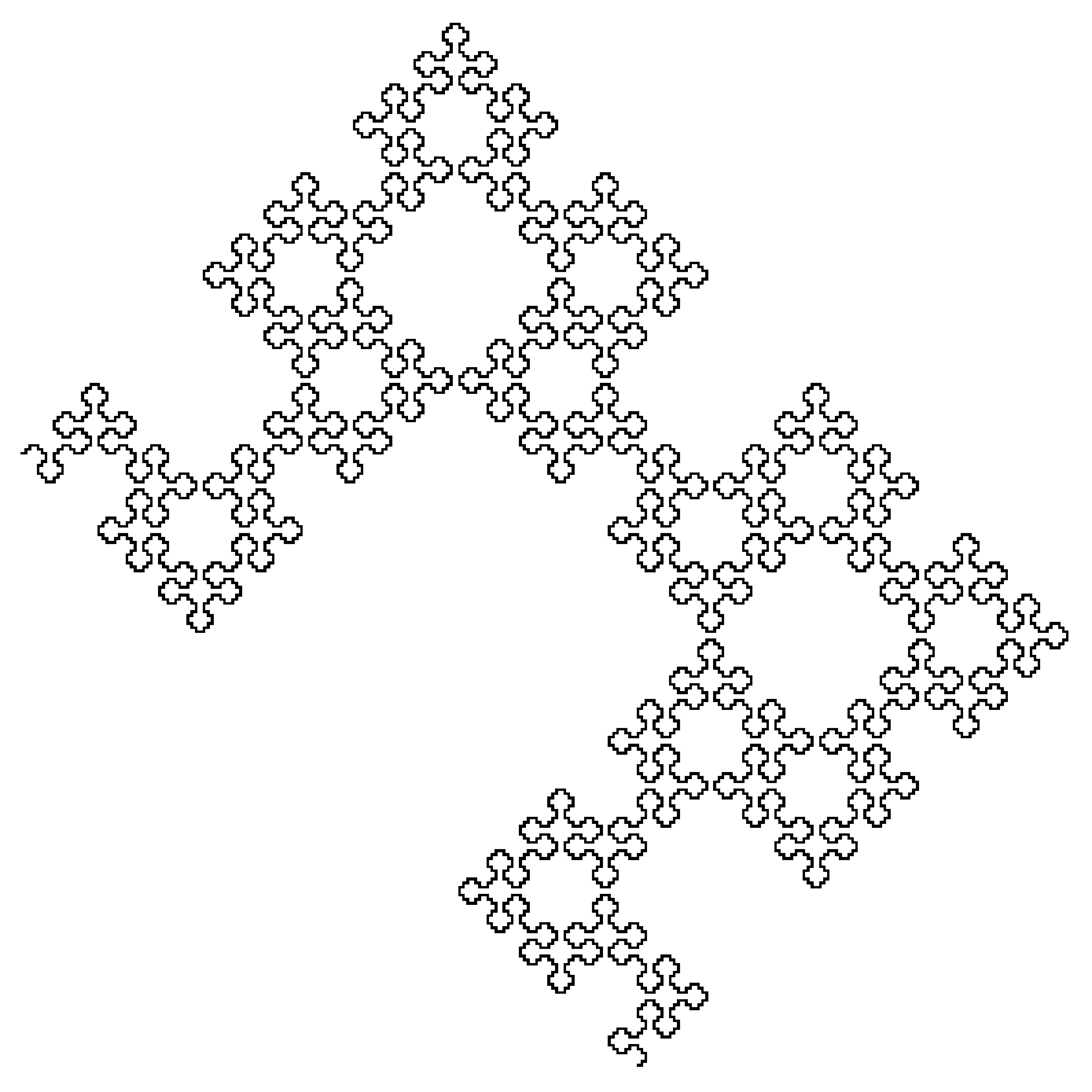} \\ \hline
\end{tabular}

\caption{The Fibonacci word curves for $i$ going from $2$ to $5$ and $n$ going from $15$ to $18$. In all of these pictures $\alpha = \frac{\pi}{2}$.}\label{fig:FibCurvesIN}
\end{center}
\end{figure}

\begin{figure}[t]
\begin{center}
\begin{tabular}{|c|c|c|} \hline
\includegraphics[width=1.25in]{./Pictures/F2_17.pdf} & \includegraphics[width=1.25in]{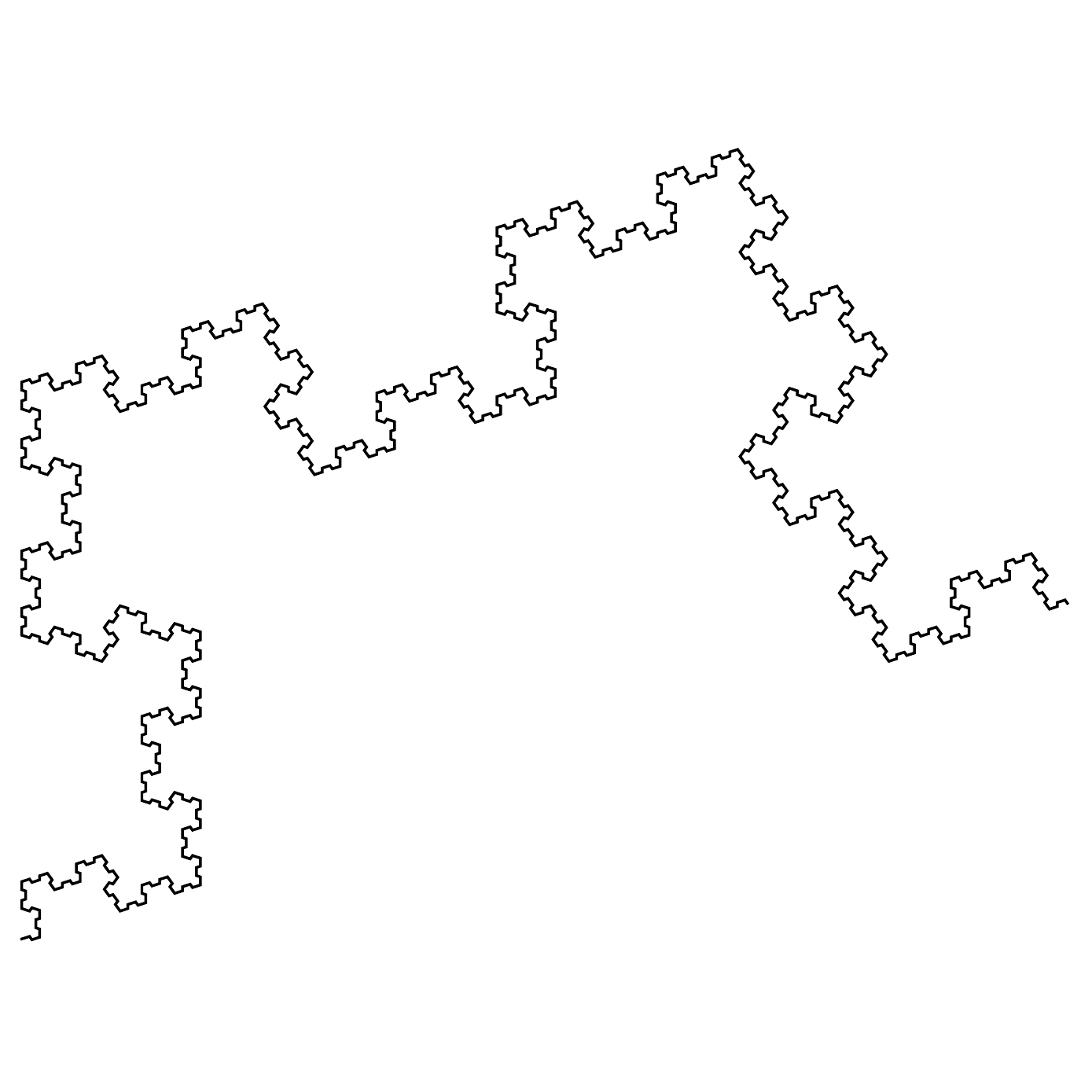} & \includegraphics[width=1.25in]{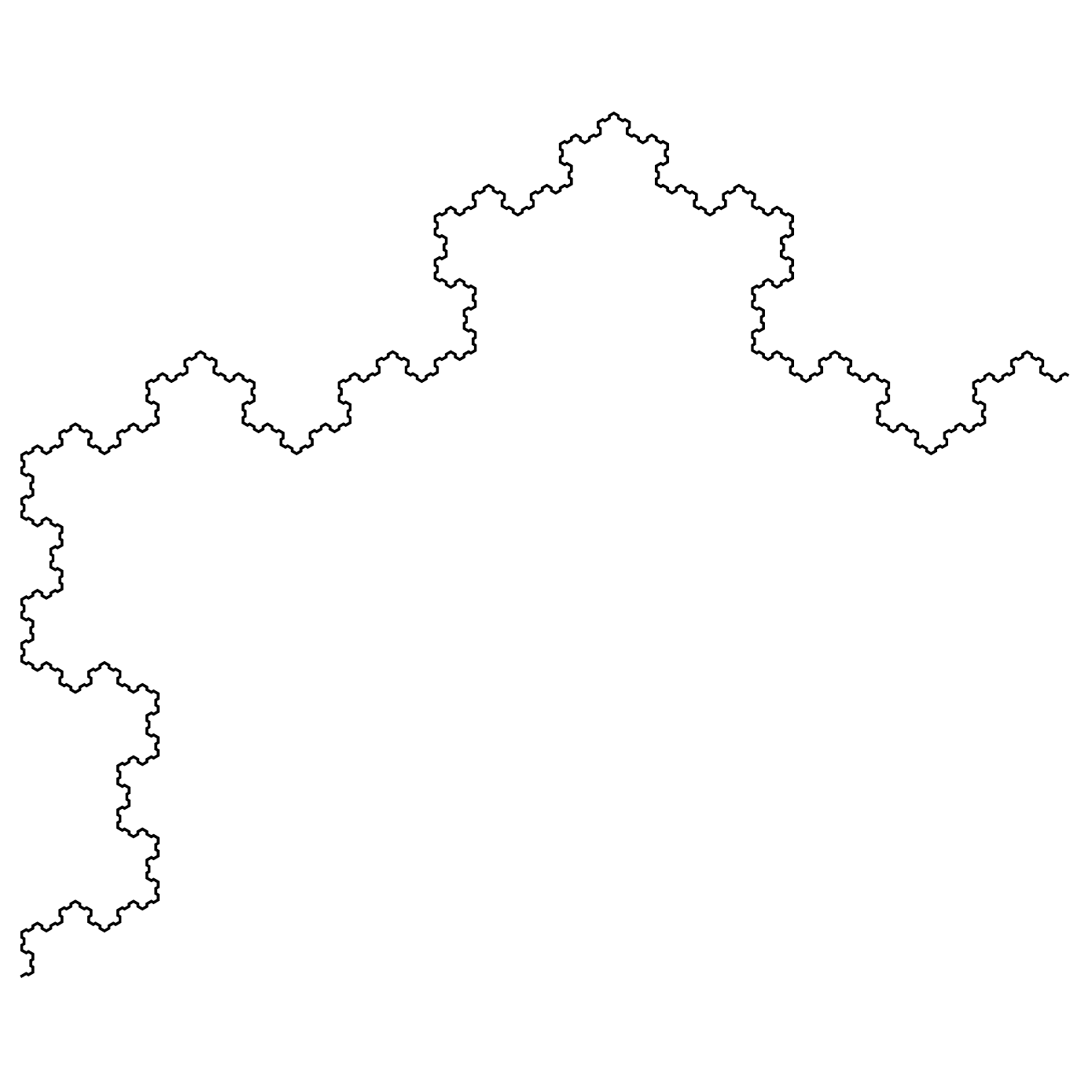}\\ \hline

\includegraphics[width=1.25in]{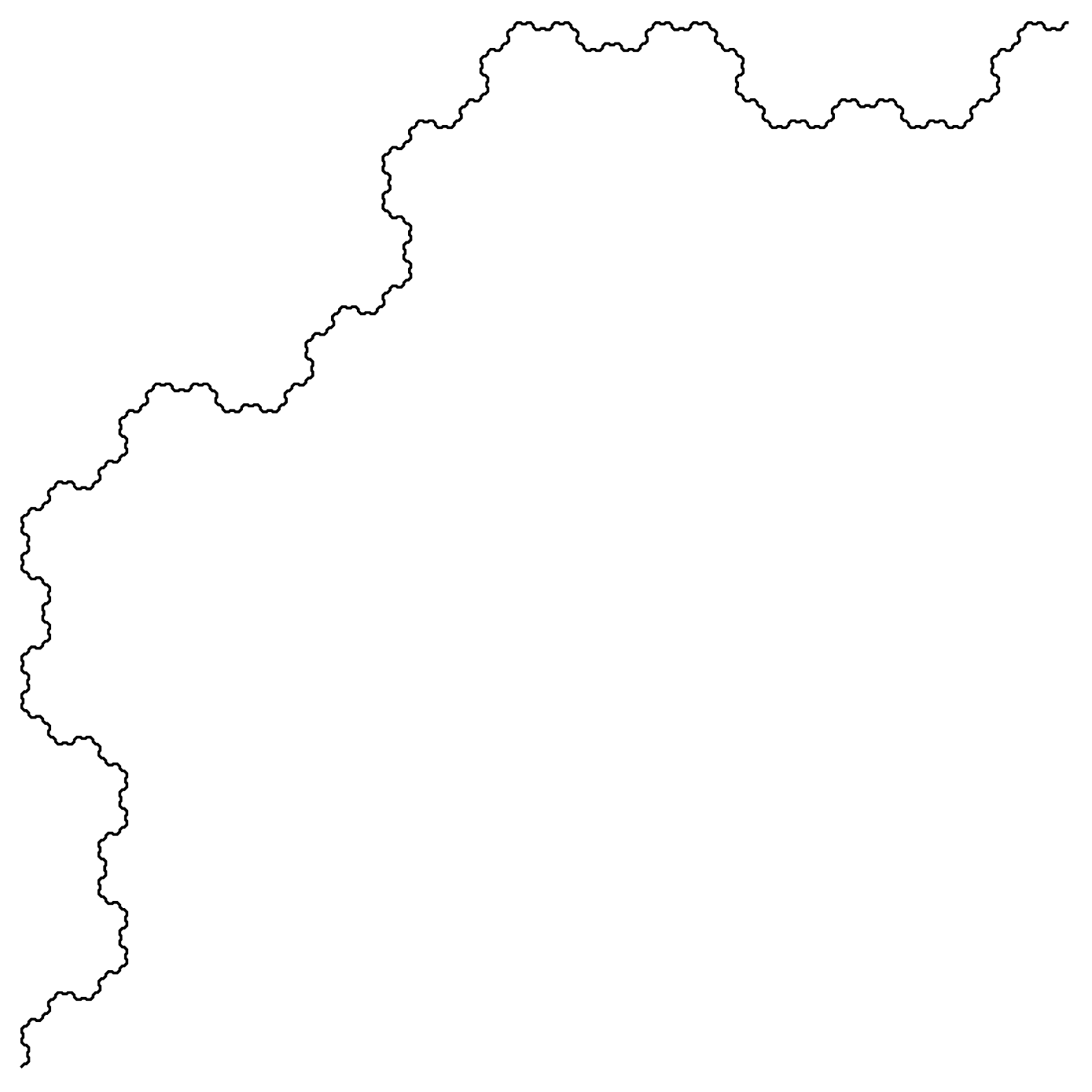} & \includegraphics[width=1.25in]{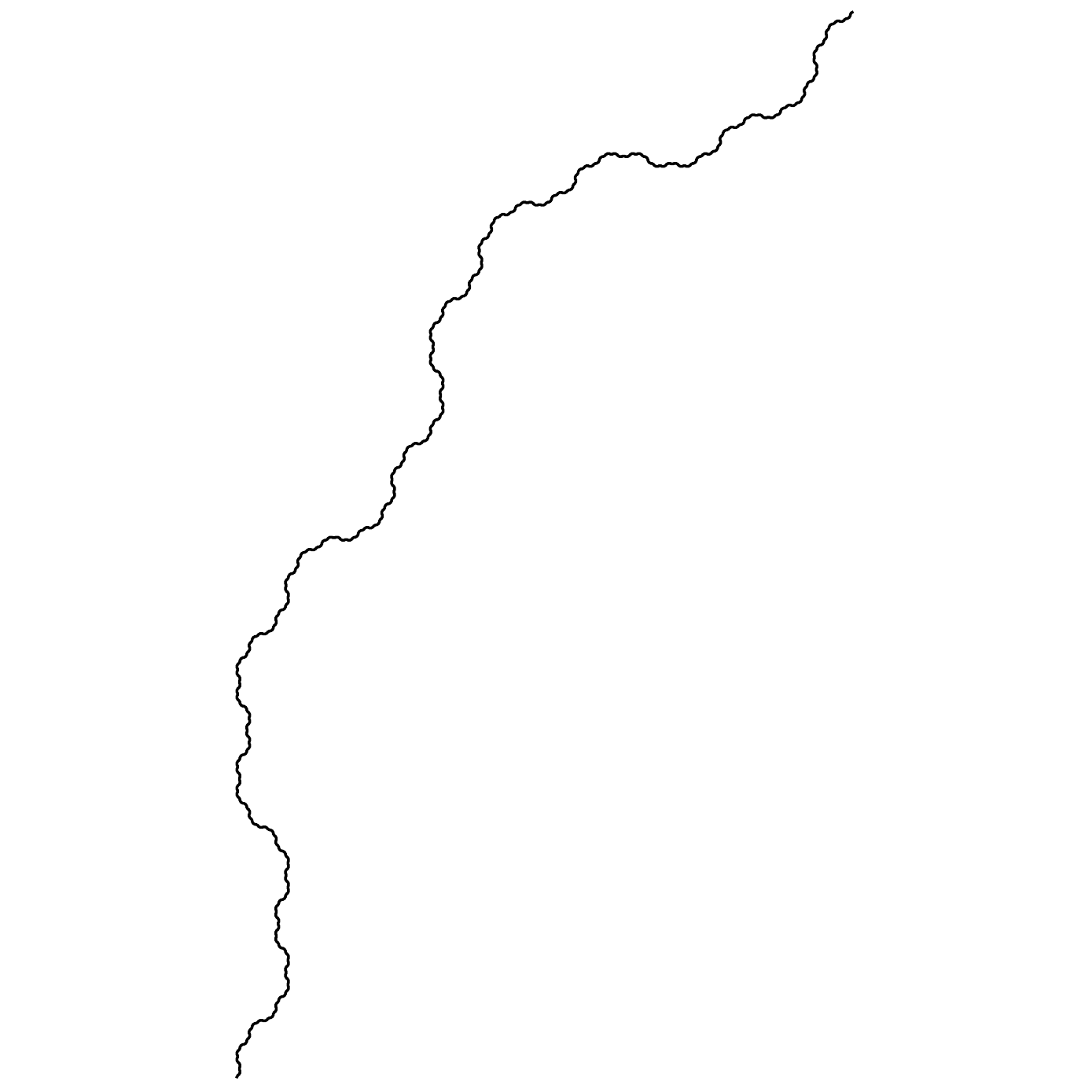} & \includegraphics[width=1.25in]{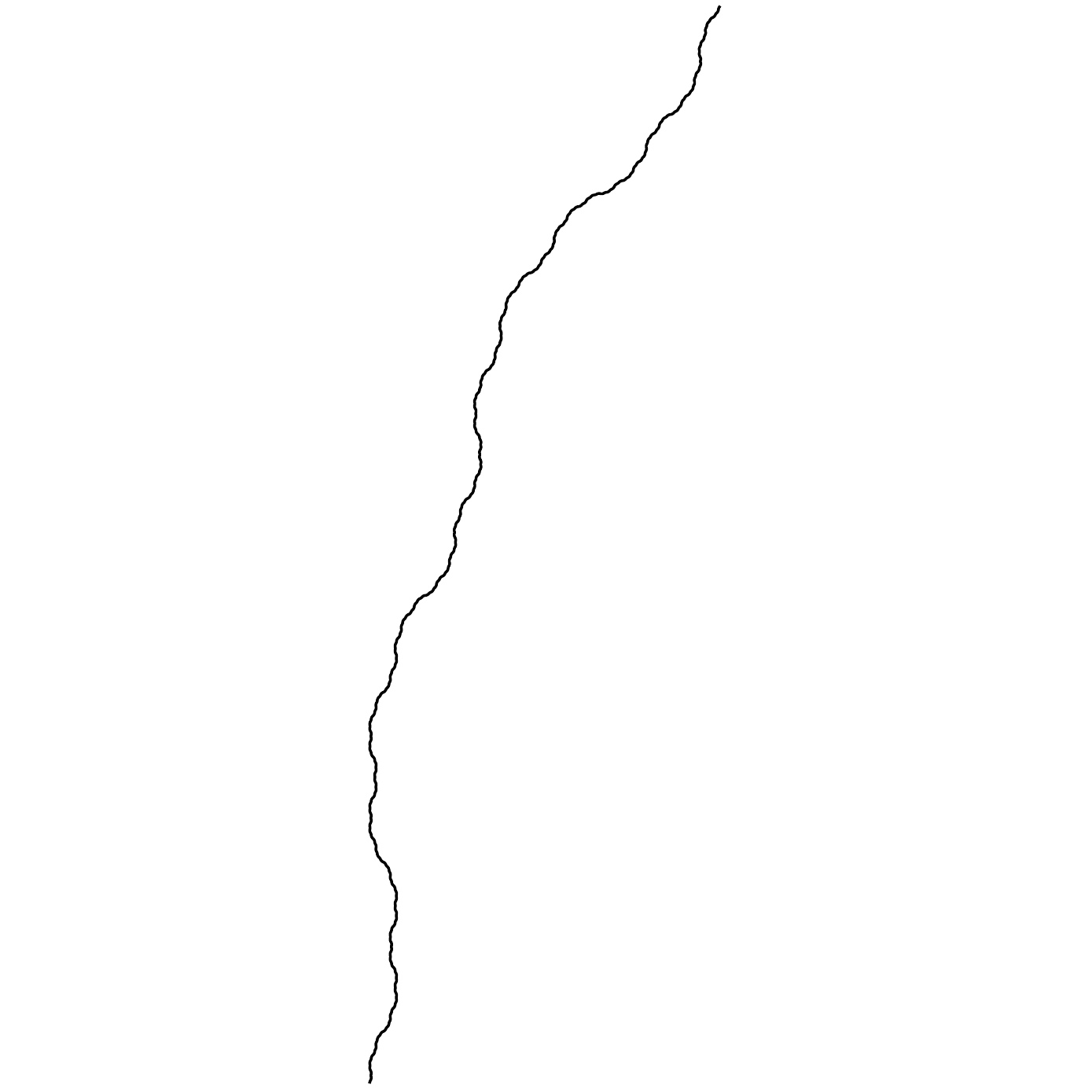} \\ \hline
\end{tabular}

\caption{The curves of $\mathcal{F}^{[2]}_{17}$ for drawing angles (left to right, top to bottom) of $2\pi/4$, $2\pi/5$, $2\pi/6$, $2\pi/8$, $2\pi/12$ and $2\pi/20$. For a drawing angle of zero, the curve is a straight vertical line.}
\label{fig:F2angle}
\end{center}
\end{figure}

\section{Fractal Geometry}\label{sec:fracgeo}

In this section we state the definitions and constructions necessary for the computation of the Hausdorff dimension of a fractal given by an iterated function system (IFS) satisfying the open set condition. We start with the definition of Hausdorff dimension. The definitions and theorems in this section are those of \cite{FalcFrac}.

\begin{definition}\label{def:HausDim}
The Hausdorff dimension, $dim_\mathcal{H}$ of a subset $F\subset\mathbb{R}^2$ is
$$dim_{\mathcal{H}}F=\inf\{s\geq 0:\mathcal{H}^s(F)=0\}=\sup\{s:\mathcal{H}^s(F)=\infty \}$$
where the Hausdorff measure $\mathcal{H}^s_\delta$ is

$$\mathcal{H}^s_\delta(F)=\inf\left\{\sum_{i=1}^\infty|U_i|^s:\{U_i\} \text{ is a $\delta$-cover of F} \right\}$$
where $\mathcal{H}^s(F)=\lim_{\delta\to 0}\mathcal{H}^s_\delta(F).$
\end{definition}

While a theoretically satisfying definition of dimension it is a difficult one to compute. But for a particular class of self-similar fractals its computation can often be done via a simpler route. 

An Iterated Function System (IFS) is an indexed collection of mappings $\{\phi_{i}\}_{i \in I}$ which in this case are from $\mathbb{R}^{2}$ onto itself. We also make the assumption that the mappings are contracting similarities. An IFS defines a set function by 
$$\mathbb{R}^{2} \supset K \mapsto \bigcup_{i \in I} \phi_i(K) = \Phi(K).$$ 
A set is called self-similar if $K = \Phi(K)$. Since similarities have a contraction ratio $r_i$ we can define a notion of dimension called variously the fractal dimension, self-similarity dimension, or similar by $dim_{SS}(K) = s$ where $s$ is the solution to $$\sum_{i \in I} r_i^{s} = 1.$$ Needless to say the computation of the self-similarity dimension from an IFS is much simpler than the direct computation of the Hausdorff dimension. However the two dimensions often take the same value for "nice" self-similar fractals.

\begin{definition}\label{def:OSC}
From \cite{FalcFrac}, an iterated function system satisfies the open set condition if
$$V\supset \bigcup_{k=1}^m\phi_k(V)$$
where $V$ is some non-empty, bounded and open set, and the $\phi_k$ are the  functions from the IFS.
\end{definition}

\begin{theorem}{\cite[Theorem 9.1]{FalcFrac}}\label{thm:attractor}
For an IFS there exists a unique attractor F, if the contractions $\{\phi_1, \phi_2, ..., \phi_k\}$ on $D\subset \mathbb{R}$ satisfy
$$|\phi_k(\vec{x})-\phi_k(\vec{y})|\le c_k|\vec{x}-\vec{y}|$$
with $(\vec{x},\vec{y})\in D$ and $c_k<1$ for each $i$.
\end{theorem}

In \cite{FalcFrac} we are given a relationship between the self-similarity and Hausdorff dimensions. 

\begin{theorem}{\cite[Theorem 9.3]{FalcFrac}}\label{thm:IFSdim}
If the IFS satisfies the open set condition, the Hausdorff dimension $dim_\mathcal{H}$ is given by solving for $s$ in the equation
\begin{equation}
	\sum_{k=1}^m c_k^s=1 \label{eq:SelfSimDim}
\end{equation}
where $c_k$ is the scaling ratio of the corresponding $\phi_k$ from the IFS.
\end{theorem}

Because the Fibonacci curves are not nativley generated by an IFS we also need a notion in which to say sequences of curves have limits other than being the fixed set of an IFS. To this end we use the Hausdorff metric. 

\begin{definition}\label{def:HausMetric}
Let $\mathcal{H}$ be the collection of non-empty compact subsets of $\mathbb{R}^{2}$. Define for $U,V \in \mathcal{H}$
$$d_{\mathcal{H}}(U,V) = \inf_{x \in U} \sup_{y \in V} |x-y| + \inf_{y \in V} \sup_{x \in U} |x-y|.$$
\end{definition}

It is a well known fact that this metric space is complete \cite{PriceFrac}. See \cite[Section 2.10]{Federer1969} for a thorough discussion of the properties of the Hausforff metric space.

\section{Fibonacci Fractals}\label{sec:fibfrac}

In this section we are concerned with showing the existence and Hausdorff dimension of the Fibonacci fractals. The fractal itself is a scaling limit of the curves $\mathcal{F}_{n}^{[i]}$ along the subsequence $n=6k+5$ for even $i$ and $n=6k+3$ for odd $i$.  But to create a self-similar description of the fractal similar to the self-similar/five-partite property of the Fibonacci words we have to include $n=6k+2$ and $n=6k$ curves as well. Let the $i$-Fibonacci curves be drawn with line segments of length one. 

\begin{definition}\label{def:scalingratio}
Consider the distance between the initial and final point of a Fibonacci curve. The ratio of this length between $k$ and $k+1$ is the scaling ratio.
\end{definition}

\begin{theorem}\label{thm:angleScale}
For drawing angle $\alpha$, the scaling ratio of the curve is
$$ R = \left((1+\cos(\alpha)) +\sqrt{(1+\cos(\alpha))^{2} + 1}\right)^{-1}.$$
\end{theorem}

\begin{figure}[t]
\begin{center}
\includegraphics[width=3in]{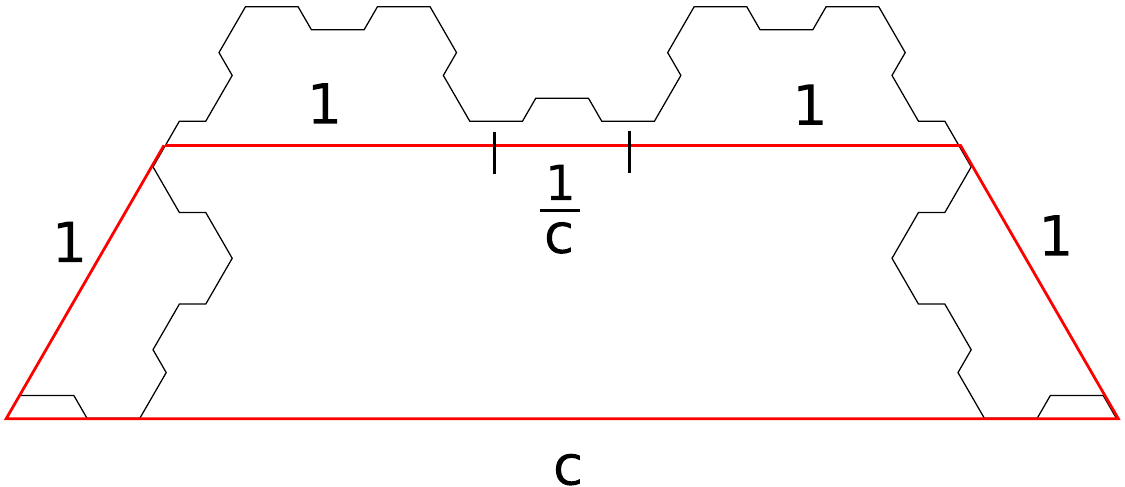}
\end{center}
\caption{For $\alpha \sim \pi/3$ a representative curve is drawn and the relationship between the widths of the five pieces of the fractal are determined to fix a scaling ratio.}\label{fig:trap}
\end{figure}

\begin{proof}
Consider the trapezoid in Figure \ref{fig:trap} describing the proportions of the curve and label the wider base of the trapezoid to have length $c$. Similarly, label the width of each of the $n-3$ sub-curves to have length 1, and finally the sub-curve $n-6$ to have length $\frac{1}{c}$. Here, $c$ is the scaling ratio from a previous curve to a larger curve. By the properties of trapezoids, the width of the larger base, $c$, can also be written as
$$c=2+\frac{1}{c}+2cos(\alpha)$$
Multiply through by $c$ to obtain a quadratic and solve via the quadratic formula:
$$c^2-2(1+cos(\alpha))c-1=0$$
$$c = \frac{2(1+\cos(\alpha)) +\sqrt{4(1+\cos(\alpha))^{2} + 4}}{2}$$
and because we are concerned with the scaling ratio from a larger curve to a smaller curve, we raise $c$ to the $-1$ and simplify to get our scaling ratio.
\end{proof}

\begin{definition}
Let $w^{[i]}_n$ be the Euclidean distance between the starting and ending vertices of $\mathcal{F}_n^{[i]}$ and $h_n$ the height of the curve in the direction perpendicular to the line connecting the first and last vertices of the curve. Notice that both $w$ and $h$ depend on the drawing angle, $n$, and $i$. The aspect ratio is defined as $\frac{w}{h}$. 
\end{definition}

Let $j=3k+2$ or $3k$ depending as $i$ is even or odd respectively. 
The scaling limit is achieved by taking the length of the individual line segments to be $w^{-1}$, thus the width of the scaled curves is always $1$. Theorem \ref{thm:CurveLimitExist} proves that the scaling limit exists as a limit in the Hausdorff metric space. The following proposition about the aspect ratio as $k$ increases will be a central piece to the proof of Theorems \ref{thm:CurveLimitExist} and \ref{thm:FracConverge}.

\begin{proposition}\label{prop:AspectRatio}
$$\lim_{k \rightarrow \infty} \frac{w^{[i]}_{j(k)}(\alpha)}{h^{[i]}_{j(k)}(\alpha)} = \frac{r_+ - 1}{\sin(\alpha)}=\frac{1-R}{R\sin(\alpha)}.  $$
\end{proposition}

\begin{proof}
Consider the two relations between the ``height'' and ``width'' of the $\mathcal{F}_{j(k)}^{[i]}$:
\begin{align}
	\label{eq:wk}w_k &= 2w_{k-1} + w_{k-2} + 2\cos(\alpha)w_{k-1}\\ 
	\label{eq:hk}h_k &= h_{k-1} + \sin(\alpha)w_{k-1}.
\end{align}
By using the width formula repeatedly in the height formula we can arrive at a more useful formula for the height. That is
\begin{align*}
	h_k &= h_1 + \sin(\alpha) \left( \sum_{q=1}^{k-1} w_q \right).
\end{align*}
Returning for the moment to (\ref{eq:wk}), by \cite[Theorem 4.10]{Niven1991} it is possible to give a closed form expression for $w_{j(k)}$ in the form $$w_{j(k)} = ar_+^{n} + br_-^{n}$$ where $r_+$ and $r_-$ are the roots of the polynomial $r^{2} - 2(1+\cos(\alpha))r - 1 = 0$. That is $$r_{\pm} = 1+\cos(\alpha) \pm \sqrt{2\cos(\alpha)}.$$ It is important to observe that $0<r_- <1 < r_+$ for $\alpha \in [0,\pi/2)$ and $r_- = r_+ = 1$ when $\alpha = \pi/2$.
Thus we can compute 
\begin{align*}
	\lim_{k \rightarrow \infty} \frac{w^{[i]}_{j(k)}(\alpha)}{h^{[i]}_{j(k)}(\alpha)} &= \lim_{k \rightarrow \infty} \frac{ar_+^{k}+ br_-^{k}}{h_1 + \sin(\alpha)\left( \sum_{q=1}^{k-1}ar_+^{q}+ br_-^{q} \right)}\\
	&=  \lim_{k \rightarrow \infty} \frac{ar_+^{k}+ br_-^{k}}{h_1 + \sin(\alpha)\left(ar_+\left(\frac{1-r_+^{k-1}}{1-r_+}\right) + br_-\left(\frac{1-r_-^{k-1}}{1-r_-}\right)\right)}\\
	&= \frac{r_+ - 1}{\sin(\alpha)}.
\end{align*}
The statement of the proposition arises from the relation between $r_+$ and $R$.
\end{proof}

\subsection{Existence}\label{ssec:existence}
There are two central existence problems that need to be addressed about Fibonacci fractals. The first is the existence of a scaling limit of the Fibonacci curves for a given $i$ and (sub)-sequence of $n$. The second is if it is possible to represent this limit curve as a self-similar set with a particular IFS. Doing so will enable the standard argument for computing Hausdorff dimension by way of computing the self-similarity (fractal) dimension.

\begin{theorem}\label{thm:CurveLimitExist}
For drawing angle $\alpha \in (0,\pi/2]$ there exists a scaling limit along a subsequence of the Fibonacci curves. For even $i$, the scaling limit of $\mathcal{F}_{6k+5}^{[i]}$ exists, for odd $i$ the scaling limit of $\mathcal{F}_{6k+3}^{[i]}$ exists. The scaling factors are  $\sqrt{2}/w_{6k+5}$ and $\sqrt{2}/w_{6k+3}$ respectively for the parity of $i$.
\end{theorem}

The choice of scaling $\mathcal{F}_{6k+5}^{[i]}$ to have width $\sqrt{2}$ is to maintain consistency with the scaling used in \cite{MonnFrac}. Any other value would of course make no difference in the following proof, however the aspect ratio would still limit to $\sqrt{2}$.

\begin{proof}
Consider the case when $i$ is even. The odd case is similar. We compute the Hausdorff distance between scaled copies of $\mathcal{F}_{6k+5}^{[i]}$ and $\mathcal{F}_{6k+11}^{[i]}$ to show that the sequence is Cauchy and thus has a limit, $\mathcal{F}^{[i]}_{\alpha}$. First $\mathcal{F}_{6k+11}^{[i]}$ is composed of $4$ copies of $\mathcal{F}_{6k+5}^{[i]}$ and a single copy of $\mathcal{F}_{6k-1}^{[i]}$, scaled so that $w_{6k+11} = \sqrt{2}$ so $$h_{6k+11} = \frac{\sqrt{2}}{w_{6k+11}}(h_{6k+5} + \sin(\alpha)w_{6k+5})$$ which for large enough $k$ is within any given $\epsilon>0$ of $1$ by Proposition \ref{prop:AspectRatio}.

Thus each of the five parts of $\mathcal{F}_{6k+11}^{[i]}$ has a bounding box with a Hausdorff distance of $\epsilon$ of the corresponding bounding box of $\mathcal{F}_{6k+5}^{[i]}$. This is because with the widths fixed to the appropriate length the bounding boxes at each level share at least one common vertex and orientation with only their heights varying. 

Now consider the portions of the curves within one of those five bounding boxes. Consider the first and last vertex of $\mathcal{F}_{5}^{[i]}$. As $\mathcal{F}_{6k+5}^{[i]}$ are recursively constructed and scaled the images of the two points has a mesh size that is no more than $\sqrt{2}^{3-k}$ which goes to zero as $k$ grows. Since the difference between $\mathcal{F}_{6k+5}^{[i]}$ and $\mathcal{F}_{6k+11}^{[i]}$ between two adjacent images of the first and last points is that while $\mathcal{F}_{6k+5}^{[i]}$ has an $\mathcal{F}_5^{[i]}$ connecting them $\mathcal{F}_{6k+11}^{[i]}$ has a $\mathcal{F}_{11}^{[i]}$. That is a bounded number of extra edges and so the distance traveled by the excursion is a bounded multiple of $\sqrt{2}^{3-k}$. For $k$ large enough this will be less than $\epsilon$. Thus the Hausdorff distance between the scaled copies of $\mathcal{F}_{6k+5}^{[i]}$ and $\mathcal{F}_{6k+11}^{[i]}$ is less than any given $\epsilon$ for $k$ large enough. Since the Hausdorff metric space is complete there exists a scaling limit set. 
\end{proof}

For $i$ even let the IFS be the five contractive similarities that map the large trapezoid onto the five similar trapezoids in Figures \ref{fig:PhiBox} and \ref{fig:IFS} such that the image of the $x-$axis falls onto the dark line. For odd $i$, use the same five maps composed with a clockwise rotation of $\pi/4$. The maps are referred to as $\phi_i$, $i=1,2,3,4,5$.

The following lemma follows directly from this definition of the IFSs by choosing $V$ to be the interior of the largest trapezoid in Figure \ref{fig:IFS}. In Proposition \ref{prop:boundOverlap} it was shown that the Fibonacci curves themselves satisfy the open set condition so that it is expected that an IFS we choose should as well. 

\begin{proposition}\label{prop:OSCsatisfied}
For both parities of $i$, the IFSs described both satisfy the open set condition.   
\end{proposition}

\begin{figure}[t]
\begin{center}
\includegraphics[width=3in]{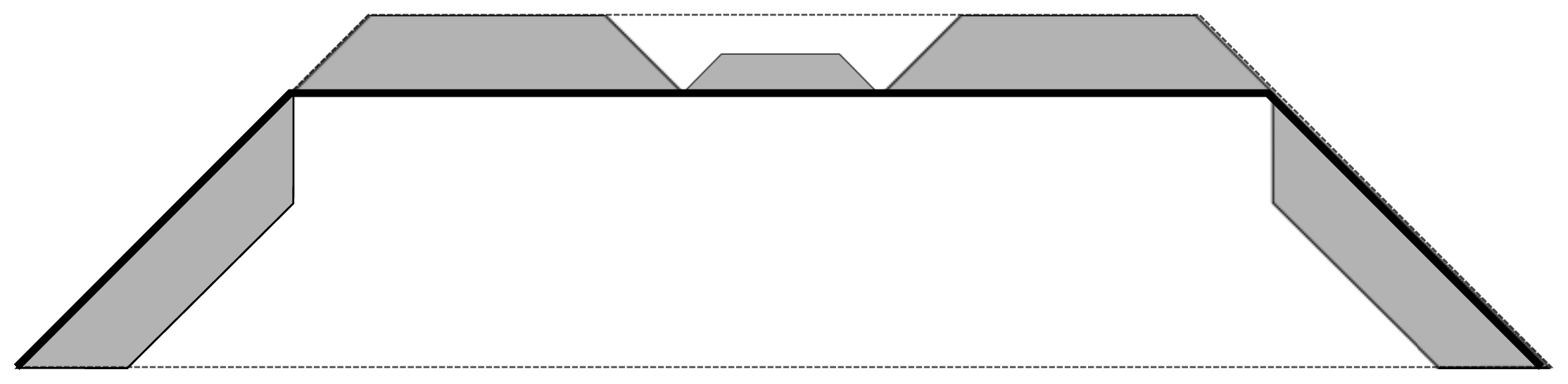}
\caption{Bounding trapezoids for $\alpha \neq \pi/2$.}
\label{fig:IFS}
\end{center}
\end{figure}

\begin{proposition}\label{prop:IFSattractor}
There exists a unique attractor for both of the above IFS for $\alpha \in (0,\pi/2]$.
\end{proposition}

\begin{proof}
Notice that all five of the $\phi$ have contraction ratio either $R^{-1}$ or $R^{-2}$ (see Theorem \ref{thm:angleScale}) both of which are less than one for all $\alpha \in (0,\pi/2]$. So by Theorem \ref{thm:attractor} there is a unique attractor.
\end{proof}

There is no a priori guarantee that $\mathcal{F}^{[i]}_{\alpha}$ is equal to the attractor of the appropriate IFS. However it is.

\begin{theorem}\label{thm:CurveIFSequiv}
The attractor of the IFS corresponding to the parity of $i$ and value of $\alpha \in (0,\pi/2]$ is the equal to $\mathcal{F}^{[i]}_{\alpha}$.
\end{theorem}

\begin{proof}
Fix an $i$ and $\alpha$. Let $\mathcal{F}^{[i]}_{\alpha}$ be the scaling limit and apply the IFS to it, i.e. consider $\bigcup_{i=1}^{5} \phi_i(\mathcal{F}^{[i]}_{\alpha})$. Because the scaling limit is already invariant under the scaling process, replacing the curve with five scaled copies of the curve so that the width remains $\sqrt{2}$ we can use Theorem \ref{thm:angleScale} to say that the scaled copies are scaled by exactly $R$ and $R^{2}$. Again since the scaling limit is as we said invariant under the scaling operation we note that Proposition \ref{prop:AspectRatio} that the aspect ratio remains unchanged. Thus the scaling operation acts on $\mathcal{F}^{[i]}_{\alpha}$ as an IFS with the five maps, the same scaling ratios, and general ``U-shaped'' pattern as the IFS defined above. 

To show that they are the same IFS we refer back to Propositions \ref{prop:6angle} and \ref{prop:curveAngle} which together imply that since $\mathcal{F}^{[i]}_{\alpha}$ has initial point at the origin that the five sub-curves are in exactly the same alignment as shown in Figure \ref{fig:IFS}. The previous paragraph shows they are the same size as well. Thus $\mathcal{F}^{[i]}_{\alpha}$ is invariant under the IFS, so by the uniqueness of the attractor $\mathcal{F}^{[i]}_{\alpha}$ is the attractor. 
\end{proof}

It should be noted what happens when $\alpha=0$. In this case the apsect ratio is $0$ and the scaling limit becomes a line segment in the direction $\alpha$ of length $\sqrt{2}$. The fractal degenerates to a line. 

With these results, we now know enough to properly calculate the Hausdorff dimension of the attractor of our IFS which in turn is the same dimension as that of the curve $\mathcal{F}^{[i]}$.



\subsection{Dimension}\label{ssec:dimension}

We are now in a position to state the first significant result of the paper, the Hausdorff dimension as a function of drawing angle. It is noteworthy that the dimension, in fact the fractals themselves do not depend on $i$ in any way. Thus the value of $i$ is only of combinatorial interest.

\begin{figure}[t]
\begin{center}
\includegraphics[width=3in]{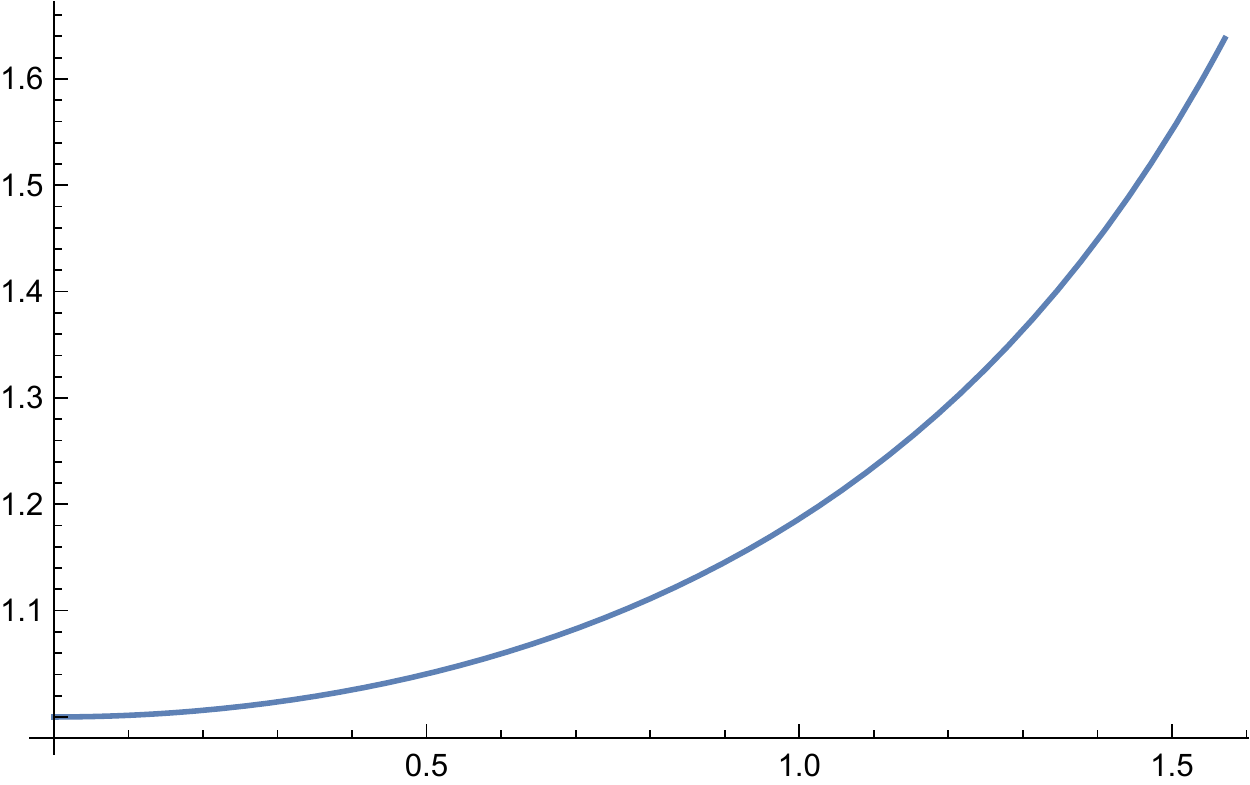}
\label{fig:dim}
\caption{The Hausdorff dimension of $\mathcal{F}_{\alpha}^{[i]}$ as a function of $\alpha$. The dimension is actually independent of $i$ whose parity only determines the exact embedding of $\mathcal{F}_{\alpha}^{[i]}$ in $\mathbb{R}^{2}$. Notice that the limit as $\alpha \rightarrow 0$ is $1$.}
\end{center}
\end{figure}

\begin{proposition}\label{prop:angleDimH}
For drawing angle $\alpha \in [0,\pi/2]$ the Hausdorff dimension of the fractal is given by
$$s=\frac{\ln(-2+\sqrt{5})}{\ln(R)}$$
\end{proposition}

\begin{proof}
Fix $i$ and $\alpha$. Here we use $R$ and $R^{2}$ to denote the scaling ratios, obtained from Theorem \ref{thm:angleScale}. These are also the scaling ratios of the IFS by Theorem \ref{thm:CurveIFSequiv}. Using the characterization of Hausdorff dimension in  Theorem \ref{thm:IFSdim} we solve for $s$ by first solving for $R^s$:
$$\sum_{k=1}^mc_k^s=1$$
$$4R^s+R^{2s}=1$$
\begin{align*}
R^s&=-2+\sqrt{5}\\
s\ln{R}&=\ln{-2+\sqrt{5}}\\
s&=\frac{\ln{-2+\sqrt{5}}}{\ln{R}}
\end{align*}
\end{proof}

\begin{corollary}\label{cor:angledim}
As the drawing angle, $\alpha$, goes to zero, the Hausdorff dimension of $\mathcal{F}_{\alpha}^{[i]}$
goes to one.
\end{corollary}

\begin{proof}
To show that the Hausdorff dimension of the curve as angle $\alpha$ changes, we observe the limit of the equation from Proposition \ref{prop:angleDimH} combined with the equation from Theorem \ref{thm:angleScale}:
$$\lim_{\alpha\to 0^+}\frac{\ln(-2+\sqrt{5})}{\ln\left(\left((1+\cos(\alpha)) +\sqrt{(1+\cos(\alpha))^{2} + 1}\right)^{-1}\right)}$$
As $\alpha$ goes to 0, the cosine factors go to 1, giving us the limit:
$$\lim_{\alpha\to 0^+}\frac{\ln(-2+\sqrt{5})}{\ln\left(\left(2 +\sqrt{5}\right)^{-1}\right)}=1.$$
\end{proof}

The last observation is that not only is the Hausdorff dimension of the Fibonacci fractals continuous in the drawing angle but also the fractals themselves. The continuity of the fractals as $\alpha$ grows above $\pi/2$ is conjectured but in that case the IFS has many overlaps that the continuity of the dimension function is not obvious. 

\begin{theorem}\label{thm:FracConverge}
The function that maps $\alpha \in [0,\pi/2]$ to the fractal $\mathcal{F}_{\alpha}^{[i]}$ for any $i \ge 2$ is continuous from above at $\alpha=0$ as a map from the interval into the Hausdorff metric space.
\end{theorem}

\begin{proof}
Consider the points $(0,0)$ and $(\sqrt{2}\sin(\alpha),\sqrt{2}\cos(\alpha))$, denote them by the set $V_0^{\alpha}$. These are the first and last points mentions in the proof of Theorem \ref{thm:CurveLimitExist}. By Theorem \ref{thm:CurveIFSequiv} $\mathcal{F}_{\alpha}^{[i]}$ can be represented by the IFS instead of the scaling limit construction. Let $V_{k}^{\alpha}$ be the image of $V_0^{\alpha}$ under $k$ many applications of the iterated function system. As $k$ goes to infinity the distance between $V_k^{\alpha}$ and $\mathcal{F}_{\alpha}^{[i]}$ goes to zero.  It thus suffices to show continuity in the Hausdorff metric of $V_k^{\alpha}$ as a function of $\alpha$ for all $k$. 

Consider the set $V_0^{\alpha}$. As a non-empty compact subset of $\mathbb{R}^{2}$ it is clearly a Hausdorff-metric continuous function of $\alpha$. Since self-similar IFSs are Hausdorff-metric continuous functions themselves we have that $V_k^{\alpha}$ depends on $\alpha$ continuously as it is a composition of continuous maps. 
\end{proof}

\bibliography{FibFrac}{}
\bibliographystyle{amsplain}

\end{document}